\documentclass[reqno]{article}
\usepackage[top=3cm, bottom=2.5cm, left=3cm, right=3cm]{geometry}

\usepackage{amssymb}
\usepackage{amsmath}
\usepackage{mathtools}
\usepackage{amsthm}
\usepackage{graphicx}
\usepackage{array}
\usepackage{titlesec}
\usepackage{eepic}
\usepackage{multicol}
\usepackage[usenames,dvipsnames]{xcolor}
\usepackage{color}
\usepackage[linktocpage=true]{hyperref}
\usepackage[hypcap=true]{caption}
\usepackage[cmtip,all]{xy}
\usepackage{enumitem}
\usepackage{leftidx}
\usepackage[dvipsnames]{xcolor}
\usepackage[mathscr]{euscript}
\usepackage{setspace}
\usepackage{needspace}

\usepackage{parskip}

\makeatletter
\def\thm@space@setup{%
  \thm@preskip=\parskip \thm@postskip=0pt
}
\makeatother

\allowdisplaybreaks

\numberwithin{equation}{section}

\newtheorem{theorem}{Theorem}[section]
\newtheorem{lemma}[theorem]{Lemma}
\newtheorem{proposition}[theorem]{Proposition}
\newtheorem{corollary}[theorem]{Corollary}

\theoremstyle{definition}
\newtheorem{definition}{Definition}[section]

\newtheorem{remark}[definition]{Remark}

\newtheorem{ltheorem}{Theorem}

\newtheorem{lcorollary}[ltheorem]{Corollary}

\newcommand{\vol}{\mathrm{vol}}

\titleformat{\subsection}[runin]{\bfseries}{}{}{}[.]
\titleformat{\subsubsection}[runin]{\bfseries}{}{}{}[.]

\makeatletter
\renewenvironment{proof}[1][\proofname]{%
   \par\pushQED{\qed}\normalfont%
   \topsep6\p@\@plus6\p@\relax
   \trivlist\item[\hskip\labelsep\bfseries#1\@addpunct{.}]%
   \ignorespaces
}{%
   \popQED\endtrivlist\@endpefalse
}
\makeatother

\addtocounter{tocdepth}{0}

\newcommand{\CC}{\mathbf{C}}

\newcommand{\RR}{\mathbf{R}}
\newcommand{\TT}{\mathbf{T}}

\newcommand{\XX}{\mathbf{X}}

\newcommand{\ZZ}{\mathbf{Z}}

\newcommand{\C}{\mathcal{C}}
\newcommand{\T}{\mathcal{T}}

\newcommand{\D}{\mathcal{D}}

\def\S{\mathcal{S}}

\newcommand{\Ee}{\mathscr{E}}

\DeclareMathOperator*\wstlim{{\wast}\mathrm{-lim}}

\DeclareMathOperator*\esssup{\mathrm{ess \, sup}}

\DeclareMathOperator*\spn{\mathrm{span}}

\DeclareMathOperator*\wstspan{\overline{\mathrm{span}^{\wast}}}

\def\Xint#1{\mathchoice
{\XXint\displaystyle\textstyle{#1}}%
{\XXint\textstyle\scriptstyle{#1}}%
{\XXint\scriptstyle\scriptscriptstyle{#1}}%
{\XXint\scriptscriptstyle\scriptscriptstyle{#1}}%
\!\int}
\def\XXint#1#2#3{{\setbox0=\hbox{$#1{#2#3}{\int}$ }
\vcenter{\hbox{$#2#3$ }}\kern-.6\wd0}}

\def\dashint{\Xint-}

\def\1{\mathbf{1}}
\def\Id{\mathrm{id}}

\def\M{\mathcal{M}}
\def\A{\mathcal{A}}

\newcommand{\llip}{\ell\mathrm{ip}}
\newcommand{\LLip}{\mathcal{L}\mathrm{ip}}
\newcommand{\bmo}{bmo}

\def\Ss{\mathcal{S}}
\def\Ps{\mathcal{P}}
\def\Ts{\mathcal{T}}

\newcommand{\Ric}{\mathrm{Ric}}

\def\BMO{\mathrm{BMO}}
\def\dom{\mathrm{dom}}
\newcommand{\Lip}{\mathrm{Lip}}

\def\B{\mathcal{B}}

\def\Ball{\mathrm{Ball}}

\def\L{\mathcal{L}}
\def\wast{{{\mathrm{w}}^\ast}}
\def\op{\mathrm{op}}
\def\cb{\mathrm{cb}}
\def\Aut{\mathrm{Aut}}

\def\SL{\mathrm{SL}}

\def\cogrowth{\mathrm{cogrowth}}

\def\R{\mathcal{R}}

\def\algtensor{\otimes_{\mathrm{alg}}}

\def\weaktensor{ \, \overline{\otimes} \, }

\title{
  H\"older classes via semigroups and Riesz transforms 
}
\date{}
\author{
  Adri\'an Gonz\'alez-P\'erez
  \thanks{
    The author was supported by the European
    Research Council consolidator grant \texttt{614195} RIGIDITY.    
    Part of this research was conducted during a stay in the ICMAT School I of the ``Thematic Research Program: Operator Algebras, Groups and Applications to Quantum Information'' in March 2019 that was funded by the long term structural funding --- Methusalem grant of the Flemish Government.
  }
}

\begin{document}

\maketitle

\begin{abstract}
  We define H\"older classes $\Lambda_\alpha$ associated with a Markovian semigroup
  and prove that, when the semigroup satisfies the $\Gamma^2 \geq 0$ condition,
  the Riesz transforms are bounded between the H\"older classes. As a consequence, this bound holds in manifolds with nonnegative Ricci curvature.
  We also show, without the need for extra assumptions on the semigroup,
  that a version of the Morrey inequalities is equivalent
  to the ultracontractivity property. This result extends the semigroup
  approach to the Sobolev inequalities laid by Varopoulos.
  After that, we study certain families of operators between the homogeneous H\"older classes. One of these families is given by analytic spectral multipliers and includes the imaginary powers of the generator, the other, by smooth multipliers analogous to those in the Marcienkiewicz theorem.
  Lastly, we explore the connection between the H\"older norm and
  Campanato's formula for semigroups. 
  
  \vskip5pt
  \noindent \textbf{Keywords.} harmonic analysis; 
  Markovian semigroups; von Neumann algebras;
  Dirichlet-spaces; metric spaces.
  
\end{abstract}

\section*{\bf Introduction}

The $\alpha$-H\"older classes $\Lambda_\alpha(X,d)$, associated with a metric space $(X,d)$.
are given by the bounded functions satisfying that the following seminorm
\begin{equation}
  \label{eq:HoelderClassical}
  \| f \|_\alpha = \sup_{x \neq y} \bigg\{ \frac{|f(x) - f(y)|}{d(x,y)^\alpha} \bigg\}
\end{equation}
is finite. Similarly, the homogeneous $\alpha$-H\"older classes are given by the (potentially)
unbounded functions satisfying that the quantity above is finite. 

The aim of this text is to extend these notions from the context of metric spaces to that of measure spaces endowed with a Markovian semigroup of operators (a geometric notion closely related to Dirichlet forms). This continues a line of research that focuses on extending techniques from from harmonic analysis that are usually formulated in terms of metric-measure spaces, into the setting of Markovian semigroups. The potential of this approach is twofold.
First, working with semigroups often allows to handle problems in harmonic analysis over large classes of spaces ---from Riemannian or subriemannian manifolds to graphs--- in a very high-level way that often avoids kernel estimates. This was recognized early on in the work on Littlewood-Paley theory in terms of semigroups developed by Stein \cite{Ste1970} ---which was written with the intention of extending several results in Harmonic analysis to Lie groups. Other applications of semigroup techniques include maximal inequalities \cite{Cowling1983}, ultracontractivity and Sobolev inequalities \cite{Va1985, VaSaCou1992}, and semigroup bounded mean oscillation spaces \cite{JunMei2012BMO, Mei2012H1BMO, JunMeiPar2014Riesz} to name just a few. 
The second advantage of semigroups is that they provide a framework that extends transparently into the noncommuative setting. There are other strategies that provide generalizations of the notion of a metric-measure space in the operator algebra context. Namely, Rieffel's compact quantum metric spaces \cite{Rief2004Compact, Rief2004Gromov, Rief2002Group} in the context of $C^\ast$-algebras and the $W^\ast$ quantum metrics of Kuperberg and Weaver \cite{KuWea2012} in the von Neumann algebraic context. Nevertheless, both approaches are more technical from a perspective of harmonic analysis. Indeed, proving that a tentative example is a compact quantum metric space can be nontrivial and, on top of that, natural notions like restrictions of functions to balls or averages over balls are not easily defined.

Recall that in $\RR^n$ with its Lebesgue measure the Laplacian operator $\Delta = - \sum_{j} \partial_{x_j x_j}^2$ generates a Markovian semigroup $T_t= e^{- t \Delta}$ (the classical heat semigroup). At an intuitive level, the action of $T_t$ is similar to averaging over balls of radius $\sqrt{t}$, i.e.
\[
  T_t \sim \dashint_{B_{\sqrt{t}}}
\]
A valuable insight is that, in many expression appearing in harmonic analysis, we can exchange the ball averages by $T_t$ without fundamentally altering the quantities we are working with. In more abstract settings, like that of Riemannian manifolds $(M,g)$ with the natural heat semigroup $T_t = e^{-t \Delta}$ generated by the Laplace-Beltrami operator $\Delta$, the idea above, that $T_t$ behaves like a metric average, often holds true, see for instance \cite[Section 4.2]{Saloff2002}. This intuition was crucial in developing the theory of semigroup bounded mean oscillation, or $\BMO$, spaces \cite{JunMei2012BMO}.

Given an abstract Markovian semigroup $T_t = e^{-t A}$ acting on a von Neumann algebra $\M$, the semigroup $P_s = e^{-s A^{1/2}}$ is also Markovian and will be called the \emph{Poisson semigroup} associated to $T_t$. The name stems from the fact that in $\RR^n$ with $A = \Delta$ the function $P_s f$ gives a solution to the Poisson equation in $\RR^{n + 1}_+ = \RR_+ \times \RR^n$. Following a classical idea that goes back at least to Stein, see \cite[Section 3.5]{Ste1970Singular}, we define the $\alpha$-H\"older seminorm associated with a semigroup $T_t$ as
\begin{equation}
  \label{eq:DefHoelderIntro}
  \| f \|_{\Lambda_\alpha^\circ}
  := \sup_{s > 0} s^{-\alpha} \bigg\| s \, \frac{d P_s}{d \, s}(f) \bigg\|_\infty,
\end{equation}
where $P_s$ is the Poisson semigroup associated to $T_t$. Recall also that, in the case of $\RR^n$, the quantity above is comparable to \eqref{eq:HoelderClassical}.

Given a semigroup $T_t = e^{- t A}$ we can define its gradient form, also known as the \textit{carr\'e du champ}, see \cite{Ledoux2000} as
\[
  2 \, \Gamma(f,g) = A f^\ast \, g + f \, A f^\ast - A (f^\ast \, g).
\]
In the classical case of $\RR^n$ with its heat semigroup the quadratic form above satisfies that
$\Gamma(f,g)(x) = \langle \nabla f(x), \nabla g (x) \rangle$. After the works of Cipriani
and Sauvageout \cite{CiSauv2003Derivations} it is known that $\Gamma(f,g)$ can be factored as $\langle \delta(f), \delta(g) \rangle$, where $\delta:\M \to \mathfrak{X}$  is an unbounded derivation into a Hilbert module $\mathfrak{X}$. The intuition here being that the Hilbert module behaves like the space of sections of the tangent bundle of a manifold (with the technicality that in this abstract setting the fibers can be infinite dimensional). There are higher order analogues of the gradient form above given by
\[
  2 \, \Gamma^{k + 1}(f,g) = \Gamma^k(A f, g) + \Gamma^k(f, A f) - A \Gamma^k(f,g).
\]
The difference is that these higher order quadratic forms are not necessarily positive. Indeed, it was noticed by P-A. Meyer that in a Riemannian manifold $\Gamma^2 \geq 0$ holds iff its Ricci curvature is nonnegative. Intuitively, these  positive curvature conditions control the growth of balls and are therefore natural in many problems on Harmonic analysis ---pretty much in the same way in which doubling spaces are natural in the study of analysis over metric-measure spaces---

In this context of abstract semigroups we can define the \emph{semigroup Riesz transform} as the Hilbert-valued operators $\R = \delta A^{-\frac12}$, which recovers the definition of the usual Riesz transform $\R = \nabla \Delta^{-\frac12}$ in the classical case. Assuming that the derivation $\delta$ satisfies a specific intertwining identity, see \eqref{dia:Commuting}, we obtain the following bound
\begin{ltheorem}
  \label{thm:Riesz}
  \emph{
  Assume that $\Gamma^2 \geq 0$. For every $0 < \alpha < 1$ we have that
  \[
    \big\| \R: \Lambda_\alpha^\circ(\Ts) \to \Lambda_\alpha^\circ(\Ts;H) \big\| < \infty.
  \]
  }
\end{ltheorem}
In Theorem \ref{thm:Riesz} above $\Lambda_\alpha^\circ(\Ts)$ is the homogeneous H\"older class defined by the seminorm \eqref{eq:DefHoelderIntro} and $\Lambda_\alpha^\circ(\Ts;H)$ is a Hilbert-valued version of the same space, see Section \ref{sct:Riesz} for the details of the definition. 
The theorem above complements earlier results that have been obtained in the literature. For instance, see \cite{Auscher2007Memoirs, CoulhonDou1999Riesz1, CoulhonDou2001Riesz2, CoulhonDou2003Noncompact} for boundedness of the Riesz transforms over the $L^p$-spaces. Recall also that a reverse bound for the Riesz transform has been obtained in \cite{JunMe2010} between semigroup $\BMO$-spaces assuming that $\Gamma^2 \geq 0$. We would also like to point out the deep consequences that dimension-free bounds for the Riesz transforms have in the study of Fourier multipliers \cite{JunMeiPar2014Riesz}.

As a straightforward corollary of the theorem above we show that, over connected and geodesically complete manifolds $(M,g)$ with nonnegative Ricci curvature, the norms $\| f \|_{\Lambda_\alpha^\circ}$ and $\| f \|_\alpha$ are comparable, where the second one is taken with respect to the path metric. Therefore

\begin{lcorollary}
  \emph{
  Let $(M,g)$ be a connected and geodesically complete $n$-dimensional manifold $(M,g)$ with $\Ric \geq 0$, we have that
  \[
    \big\| \nabla \cdot \Delta^{-\frac12}: \Lambda_\alpha^\circ(M,d) \to \Lambda_\alpha^\circ(M,d; \ell^2_n) \big\|
    < \infty.
  \]
  }
\end{lcorollary}

The theorem above is already known and can be proved with non-semigroup related techniques by first using kernel estimates to express $\R = \nabla \cdot \Delta^{-\frac12}$ as a Calder\'on-Zygmund operator and then using that Calder\'on Zygmund are bounded over H\"older classes in the context of doubling metric spaces, see also \cite{PepeGatto2005}. Although the result is known, the advantage of our method is that the proof is immediate and does not require kernel estimates. All the complexity of the result above is thus ``hidden'' inside the definition of $\Lambda_\alpha^\circ(\Ts)$ and explicit kernel bounds only make appearance when comparing the norm of $\Lambda_\alpha^\circ(\Ts)$ with the classical H\"older norm.

Next, we show that ---in the spirit of the Hardy-Littlewood-Sobolev inequalities developed by Varopoulos \cite{Va1985, VaSaCou1992}--- satisfying an analogue of the Morrey inequality is equivalent to satisfying an ultracontractivity property. Recall that Morrey's inequality is a form of Sobolev embedding which says that, for some domain $\Omega \subset \RR^n$, when $p > n$, the Sobolev space $W^{p,1}(\Omega)$ embeds into the H\"older classes of exponent $(1 - n /  p)$. We obtain the following result.

\begin{ltheorem}
  \label{thm:Morrey}
  \emph{
  Let $\T = (T_t)_{t \geq 0}$ be a Markovian semigroup and $A$ its generator.
  The following are equivalent
  \begin{enumerate}[leftmargin=1cm, label={\rm (\roman*)}, ref = {\rm (\roman*)}]
    \item \label{itm:Morrey1} $\T$ satisfies ultracontractivity property \hyperref[eq:Ultracontractivity]{\rm ($R_n$)}, i.e.
    \[
      \big\| T_t: L_p(\M) \to L_q(\M) \big\| \lesssim \frac{1}{t^{\frac{n}{2}\big\{\frac1{p} - \frac1{q} \big\}}}
    \]
    for every $1 \leq p < q \leq \infty$.
    \item \label{itm:Morrey2} The following Morrey inequality holds
    \[
      \| f \|_{{\Lambda_{1 - \frac{n}{p}}^\circ}} \lesssim \big\| A^\frac12 f \big\|_p
    \]
    for some $p > n$.
  \end{enumerate}
  }
\end{ltheorem}
The proof of Theorem \ref{thm:Morrey} above is completely elementary but nevertheless does not seem to be known, at least to the author's knowledge. It can also be interpreted as further evidence of the usefulness of the semigroup H\"older classes defined by \eqref{eq:DefHoelderIntro}.

Our next result gives bounds over H\"older spaces for certain families of smooth Fourier multipliers that behave like classical Marcinkiewicz multipliers in $\RR^n$. Let $G$ be a locally compact group and $\lambda: G \to \B(L^2 G)$ be its left regular representation $\lambda_g\xi(h) = \xi(g^{-1} h)$. We can define the (reduced) group von Neumann algebra of $G$, $\L G$ as the von Neumann algebra generated by $\{\lambda_g\}_{g \in G}$. Given $m: G \to \CC$ an operator $T_m$ is said to be the Fourier multiplier of symbol $m$ iff it acts by linear extension of the map $\lambda_g \mapsto m(g) \lambda_g$. Notice that, when $G$ is Abelian, $\L G$ is isomorphic to the algebra of essentially bounded functions over the Pontryagin dual of $G$, $L^\infty(\widehat{G})$, and $T_m$ is indeed a Fourier multiplier in the classical sense.

Since we want to obtain smooth multiplier theorems assume for the sake of simplicity that $G$ is a Lie group. This assumption is not necessary and the Theorem \ref{thm:MarcikiewiczDual} inside Section \ref{sct:Multipliers} will be formulated in more generality. Fix $\XX = \{X_1, X_2, ..., X_r\}$ a generating system of vector in its Lie algebra (what is usually called a H\"ormander system). We will denote by $D_0(\XX)$ the local dimension associated to $\XX$ given by
\[
  D_0(\XX) = \sum_{j = k}^\infty k \, \dim ( F_k / F_{k - 1}),
\] 
where $F_0  = \{0\}$ and $F_{k +1} = \spn \{ F_k, [F_k, \XX]\}$. Just for intuition recall that the balls or radius $r$ with respect to the subriemannian metric generated by $\XX$ grow like $r^{D_0(\XX)}$ for small $r > 0$.  Recall also that we have an associated hypoelliptic operator
\[
  \Delta_\XX = - \sum_{j = 1}^r X_j^2,
\]
with respect to which we can measure smoothness.

We also need natural Markovian semigroups on $\L G$. We will use semigroups of Fourier multipliers, which are inevitably given by linear extension of $T_t(\lambda_g) = e^{-t \psi(g)} \lambda_g$. It is possible to characterize the functions $\psi:G \to \RR_+$ that generate a Markovian semigroup as the conditionally negative functions. We define the homogeneous H\"lder classes $\Lambda_\alpha^\circ(\L G)$ with respect to the Poisson semigroup, given by
\[
  P_s(\lambda_g) = \exp(-s \, \psi(g)^\frac12) \, \lambda_g.
\]
We have that following result
\begin{ltheorem}
  \label{thm:corMarcikiewiczDual}
  \emph{
   Let $\XX = \{X_1, ..., X_r\}$ be a H\"ormander system of right-invariant vector fields over $G$,
   a unimodular Lie group and $\psi: G \to \RR_+$ a conditionally negative function. We have that,
   for every $0 < \alpha < 1$
   \[
    \big\| T_m: \Lambda^\circ_\alpha(\L G) \to \Lambda^\circ_\alpha(\L G) \big\|
    \lesssim_{(\alpha)} \sup_{t > 0} \Big\{ \big\| m \, \eta( t \psi^\frac12) \big\|_{W^{2,s}_{\XX}(G)} \Big\}
   \]
   where $s > D_0(\XX)/2$, the Sobolev norm of $W^{2,s}_\XX(G)$ is given by
   \[
     \| f \|_{W^{2,s}_\XX(G)} = \| (\1 + \Delta_\XX)^{s/2} f \|_2
   \]
   and $\eta$ is a cut-off function.
   }
\end{ltheorem}
Recall that Theorem \ref{thm:MarcikiewiczDual} can be understood as an extension to the H\"older case of \cite[Theorem C]{GonJunPar2015}.

Lastly, we explore analogues of a characterization of H\"older functions due to Campanato \cite{Campanato1963}. In $\RR^n$ it holds that a function is H\"older iff for every $r$ and some $1 \leq p < \infty$
\[
  \bigg( \dashint_{B_r(x)} \big| f(z) - f_{B_r(x)} \big|^p d z \bigg)^\frac1{p}
  \, \leq \, 
  C \, r^{\alpha},
\]
where $f_{B}$ represents the average integral of $f$ over some set $B$. Furthermore, the optimal constant $C$ in the inequality above being comparable to $\| f \|_\alpha$. In the spirit of \cite{JunMei2012BMO}, we introduce the seminorms
\begin{eqnarray*}  
  \| f \|_{\LLip_\alpha}
    & = & \sup_{s > 0} \left\{ s^{-\alpha} \, \big\| P_s | f - P_s f|^2 \big\|_\infty^\frac12 \right\}, \\
  \| f \|_{\llip_\alpha}
    & = & \sup_{s > 0} \left\{ s^{-\alpha} \, \big\| P_s|f|^2 - |P_s f|^2 \big\|_\infty^\frac12 \right\}
\end{eqnarray*}
and show that, when $\Gamma^2 \geq 0$ and $\alpha$ small, they are comparable to each other and they are majorized by the seminorm of \eqref{eq:DefHoelderIntro}. We left open the problem of whether the reverse inequality holds. Although open in general, the results in Section \ref{sct:Riesz} allow us to prove the equivalence in the semicommutative case and \emph{a posteriori} in the context of quantum Euclidean spaces in the sense of \cite{GonJunPar2017singular}. One of the applications of the Campanato's formula above is that it allows to prove boundedness for Calder\'on-Zygmund operators over H\"older classes in a transparent way. In a follow up paper \cite{Gon2019CZHoelder}, we will concern ourselves with proving the norm equivalence of the norms above in the case of quantum Euclidean spaces and with the boundedness of the noncommutative Calder\'on-Zygmund operators introduced by the author with Junge and Parcet in \cite{GonJunPar2017singular} between $\alpha$-H\"older classes.

\section{\bf Preliminaries \label{sct:Preliminaries}}

\subsection*{Von Neumann algebras}
Through this article we will work with von Neumann algebras. The interested reader can read more on the topic in standard texts like \cite{Blackadar2006Book, TaI, TaII}. A \emph{von Neumann algebra} is a sub $\ast$-algebra $\M$ of $\B(H)$ that is closed in the weak operator topology. Among the different characterization that these algebras admit one specially interesting is that they are dual Banach spaces. We will denote its predual by $\M_\ast$. We recall also that an element in $\M$ is called positive if it can be expressed as $g^\ast g$ and that the positive cone $\M_+ \subset \M$ induces a partial order. The intuition behind these algebras is that they behave like noncommutative analogues of $L^\infty(\Omega)$ and that measure-theoretical concepts can be generalized into von Neumann algebras. We will say that a linear operator $\tau: \M_+ \to [0,\infty]$ is a \emph{trace} if $\tau(u f u^\ast) = \tau(f)$ for every unitary $u \in \M$. We will say that $\tau$ is 
\begin{itemize}
  \item \emph{normal}, if $\tau$ is weak-$\ast$ continuous or, alternatively, is $\sup_\alpha \tau(f_\alpha) = \tau(f)$ for any increasing sequence of positive operators with $f = \sup_\alpha f_\alpha$.
  \item \emph{semifinite}, if for every $f \in \M_+$ with $\tau(f) = \infty$,
  there is a $g \leq f$ with $\tau(g) < \infty$.
  \item \emph{faithful}, if $\tau(f) = 0$ implies that $f = 0$.
\end{itemize}
If a trace $\tau$ is normal, semifinite and faithful we will say that it is a \emph{n.s.f.} trace. The von Neumann algebras that admit such trace are called semifinite. The intuition here is that traces behave like measures and that the properties above are natural generalizations of $\sigma$-additivity, semifiniteness and faithfulness respectively.

We will barely use the theory of noncommutative $L_p$-spaces in this text since most of our results can be formulated using the norm of $\M$. Nevertheless, some conditions on $L^p$ would make its appearance in Section \ref{sct:Riesz}, when dealing with Morrey's inequality in Section \ref{sct:Ultrcontractivity} and at the end of Section \ref{sct:Campanato}. Let $\S_\tau \subset \M$ be the dense ideal spanned by projection $p \in \M$ such that $\tau(p) < \infty$. The spaces $L^{p}(\M, \tau)$, for $1 \leq p < \infty$, are defined as the abstract closure of $\S_\tau$ with respect to the norm
\[
  \| f \|_p = \tau \big( |f|^p \big)^\frac1{p} = \tau \big( (f^\ast f)^\frac{p}{2} \big)^\frac1{p}
\]
We will use the convention that $L^\infty(\M,\tau) = \M$ and denote its norm by $\| \cdot \|_\infty$. There is an alternative construction of $L^p(\M,\tau)$ as subspaces of the algebra $L^0(\M,\tau)$ of $\tau$-measurable operators. The interested reader can look more on $L^p$-spaces on \cite{PiXu2003,Terp1981lp}. It is also worth noticing that we can identify $L^1(\M)$ isometrically with the predual $\M_\ast$. Similarly $L^2(\M,\tau)$ is isomorphic to the Gel'fand-Neumark-Segal construction associated to $\tau$. When $\tau$ is clear from the context we shall omit the dependency of $L^p(\M)$  on it.

\subsection*{Markovian operators and semigroups}
Here we are going introduce the basic for Markovian semigroups over semifinite von Neumann algebras. We will start with the definition of a Markovian operator. In the definition bellow $M_m[\M]$ represents the matrices with entries in $\M$. Notice that if $\M \subset \B(H)$, then $M_m[\M] \subset \B(H^{\otimes m})$.

\begin{definition}
  \label{def:MarkovOperator}
  Let $(\M,\tau)$ be a pair of a von Neumann algebra and a n.s.f. trace. An operator $T: \M \to \M$ is said to be \emph{Markovian} iff
  \begin{enumerate}[leftmargin=1.5cm, label={\rm (\Roman*)}, ref={\rm (\Roman*)}]
    \item \label{itm:markovOp.1} $T$ is unital and completely positive, i.e. $\Id_{\M_m} \otimes T: M_m[\M] \to M_m[\M]$
    preserves the positive cone for every $m \geq 1$ and $\tau(\1) = \1$.
    \item \label{itm:markovOp.2} $T$ is $\tau$-preserving, i.e. $\tau(T(f)) = \tau(f)$, for every $f \in \M_+$.
    \item \label{itm:markovOp.3} $T$ is weak-$\ast$ continuous.
    \item \label{itm:markovOp.4}
    $T$ is self-adjoint $\tau(f^\ast \, T g) = \tau(T f^\ast \, g)$, for every $f,g \in \M \cap L^2(\M)$.
  \end{enumerate}
\end{definition}

Recall that in the definition above there is some redundancy between the different properties. For instance, the fact that $T(\1) = \1$ can be deduced from the fact that $T$ is self adjoint, $\tau$-preserving and weak-$\ast$ continuous. We would like to point out that in part the literature on Markovian semigroups the property \ref{itm:markovOp.4} is called symmetry. We will also point out that when $\tau(T(f)) \leq \tau(f)$, the operator is called \emph{submarkovian}. 

The reason why we impose that all the matrix amplification of $T$ should preserve the positive cone is that complete positivivty is required for the Kadison-Schwarz inequality
\begin{equation}
  \label{eq:Kadison}
  \tag{K}
  Tf^\ast \, T f \leq T \big( f^\ast f \big).
\end{equation}
to hold, see \cite[Proposition 1.5.7(1)]{BroO2008}.

\begin{definition}
  \label{def:MarkovSemigroup}
  A family of operators $\Ts = (T_t)_{t \geq 0}$ is said to be a semigroup if
  $T_0 = \Id$ and $T_t \circ T_s = T_{s + t}$. $\Ts$ is said to be a \emph{Markovian semigroup}
  if each of the operators $T_t:\M \to \M$ is Markovian and
  the map $t \mapsto T_t$ is pointwise weak-$\ast$ continuous
\end{definition}

It is easy to see that, by property \ref{itm:markovOp.2}, the operators $T_t$ can be extended to the space $L^2(\M)$, an indeed to all $L^p$-spaces, and that they are norm continuous in $L^p$ for $p < \infty$. Using the spectral theorem, this implies that there is an unbounded, self-adjoin and positive operator $A$ in $L^2(\M)$ such that $T_t = e^{- tA}$. We will refer to $A$ as the \emph{infinitesimal generator} or simply as the generator of $\Ts$. This operator is densely defined and closable and its domain in $L^2(\M)$ is given by
\[
  \dom_2(A) = \bigg\{ f \in \M : \lim_{t \to 0^+} \frac{f - T_t f}{t} \mbox{ exists } \bigg\},
\]
where the limit is meant to exists in the norm topology. In $\M$ the situation is similar but with the norm topology of $L^2(\M)$ replaced by the weak-$\ast$ topology of $\M$. In that case we have that $A$ is weak-$\ast$ closable and has a weak-$\ast$ dense domain that we will denote by $\dom(A)$ or $\dom_\infty(A)$. 

An observation that will be used throughout the article is that if $A$ is the infinitesimal generator of a Markovian semigroup, then so is $A^\beta$, where $0 < \beta \leq 1$. When $\beta = 1/2$ the Markovian semigroup generated will be called the \emph{Poisson semigroup} $\Ps = (P_s)_{s \geq 0}$ associated to $\Ts$. 

The following subordination formula, which can be traced back to Stein \cite{Ste1970}, has been extensively used in the semigroup literature.

\begin{lemma}[{\bf (Subordination Formula)}]
  \label{lem:Subordination}
  \
  \begin{enumerate}[leftmargin=1cm, label={\rm (\roman*)}, ref={\rm (\roman*)}]
    \item \label{itm:Subordination.1} 
    Let $\Ts$ be a Markovian semigroup and $\Ps = (P_s)_{s \geq 0}$ its associated Poisson semigroup
    it holds that for every $f \in L^2(\M)$
    \begin{equation}
      \label{eq:Subordination}
      P_s(f) = \frac1{2 \sqrt{\pi}} \int_0^\infty s e^{-\frac{s^2}{4 v}} v^{-\frac{3}{2}} T_v(f) \, d v.
    \end{equation}
    We will denote by $\phi_s(v)$ the function we are integrating $T_v$ against.
    \item \label{itm:Subordination.2}
    The function $\phi_s(v)$ above satisfies that, for every $k \geq 0$
    \[
      \sup_{s>0} \Big\| s^k \, \frac{\partial^k}{\partial \, s^k} \phi_s(v) \Big\|_1 < \infty.     
    \]
  \end{enumerate}
\end{lemma} 

\begin{proof}
  The proof of both assertions is trivial. The first follows from the following explicit commutation of the inverse Laplace transform of $r \mapsto e^{- s r^\frac12}$:
  \[
    e^{- s r^\frac12} = \frac1{2 \sqrt{\pi}} \int_0^\infty s e^{-\frac{s^2}{4 v}} v^{-\frac{3}{2}} e^{-v r} \, d v
  \]
  and an application of the spectral calculus in $L^2(\M)$ Point \ref{itm:Subordination.2} is a straightforward computation.
\end{proof} 

\begin{corollary}
  \label{cor:Subordination}
  \
  \begin{enumerate}[leftmargin=1cm, label={\rm (\roman*)}, ref={\rm (\roman*)}]
    \item \label{itm:Cor.Subordination.1}
    The formula in \eqref{eq:Subordination} holds for every $x \in \M + L^1(\M)$.
    \item \label{itm:Cor.Subordination.2}
    For every fixed $s > 0$ and integer $k \geq 0$ the operator given by
    \[
      f \mapsto s^k \, \frac{d^k}{d \, s^k} P_s(f)
    \]
    is weak-$\ast$ continuous over bounded sets of $\M$.
    \item \label{itm:Cor.Subordination.3}
    For every $k \geq 0$ the maps above are uniformly bounded in $s$ for $f \in L^p(\M)$ and $1 \leq p \leq \infty$.
  \end{enumerate}
\end{corollary}

\begin{proof}
  For \ref{itm:Cor.Subordination.1} just notice that $x \mapsto T_v(x)$ is a contraction over all $L^p$-spaces and use the fact that $|\phi_s(v)|$ is uniformly in $L^1(\RR_+)$.
  
  For \ref{itm:Cor.Subordination.2}, take a bounded net $x_\lambda$ such that $x_\lambda \to x$ in the weak-$\ast$ topology. Notice that for any $\varphi \in \M_\ast$
  \[
    \esssup_{v} \big| \big\langle \varphi, T_v(x_\lambda - x) \big\rangle \big| \leq 2 \, \| \varphi \|_{\ast} \, \Big( \sup_{\lambda} \| x \|_\infty \Big).
  \]
  We also have that for every fixed $v$, $\langle \varphi, T_v(x_\lambda - x) \rangle$ goes to $0$ with $\lambda$. Applying the dominated convergence theorem gives
  \begin{eqnarray*}
    \Big| s^k  \frac{d^k}{d \, s^k} P_s(x_\lambda - x) \Big|
      & = & \Big| \int_0^\infty s^k \frac{\partial^k \phi_s(v)}{d s^k} \, \langle \varphi, T_v(x_\lambda - x) \rangle d\, v \Big|\\
      & \leq & \int_0^\infty \Big| s^k \frac{\partial^k \phi_s(v)}{d s^k} \Big| \, \big| \langle \varphi, T_v(x_\lambda - x) \rangle \big| \, d\, v \longrightarrow 0 
  \end{eqnarray*} 
\end{proof}

\subsection*{Operator spaces}
We will use (very lightly) the theory of operator spaces. Recall that an operator space is a closed linear subspace $E \subset \B(H)$. For any such space we can define their matrix amplifications $M_m[E] \subset \B(H^{\otimes m})$ as the matrices over $\B(H)$ whose entries lay in $E$. The morphisms in the category of operator spaces are the linear maps $\phi: E \to F$ such that all of its matrix amplification $(\Id \otimes \psi): M_m[E] \to M_m[F]$ are uniformly bounded, i.e.
\[
  \| \phi \|_{\cb} = \sup_{m \geq 1} \Big\{ \big\| \Id \otimes \psi: M_m[E] \to M_m[F] \big\| \Big\} < \infty.
\]
Those maps are called \emph{completely bounded}. The interested reader can look up the different books that already exist on the topic like \cite{Pi2003, EffRu2000Book} or \cite{BleMer2004Operator}. Operator spaces can be intrinsically characterized by a collection of matrix norms over $M_m[E] = M_m(\CC) \otimes E$ that satisfy the so-called Ruan's Axioms. We will called this extra structure, described by either the collections of matrix norms or by the isometric embedding into $\B(H)$ the \emph{operator space structure} of $E$. One single Banach space $E$ can have many nonisomorphic operator space structures. One interesting example, that will appear throughout the article, specially around Section \ref{sct:Campanato}, is that of a Hilbert space.
There are two easily described operator space structures over a Hilbert space $\ell^2$, the \emph{column one} and the
\emph{row one} which are given by
\begin{eqnarray*}
  \Big\| \sum_{n} x_n \otimes e_n \Big\|_{M_m[\ell^2_c]} & = & \bigg\| \sum_{j = 0}^\infty x_j^\ast x_j \bigg\|_{M_m(\CC)}\\
  \Big\| \sum_{n} x_n \otimes e_n \Big\|_{M_m[\ell^2_r]} & = & \bigg\| \sum_{j = 0}^\infty x_j x_j^\ast \bigg\|_{M_m(\CC)}
\end{eqnarray*}
where $(x_n)_n$ is a sequence of elements in $M_m(\CC)$ and $e_n$ is the canonical base of $\ell^2$. In the construction above we can change $\ell^2$ by any Hilbert space $H$ and $M_m[\CC]$ by any von Neumann algebra $\M$. We will denote those spaces by $\M[H^c]$ and $\M[H^r]$ respectively. We mention also that these spaces are examples of $W^\ast$-Hilbert modules over $\M$, see \cite{Maunilov2005Hilbert, Lance1995}.

\section{\bf Basic definitions \label{sct:Definition}}

The following definition goes back all the way back to Stein \cite[Section 4.2]{Ste1970Singular}
in the case of $\RR^n$ with the usual heat semigroup,

\begin{definition}
  Let $\Ts = (T_t)_{t \geq 0}$ be a Markovian semigroup and $\Ps = (P_s)_{s \geq 0}$ its
  associated Poisson semigroup. For $0 < \alpha < 1$, we define the
  \emph{semigroup $\alpha$-H\"older classes} $\Lambda_\alpha(\T)$ as the subsets of
  $\M$ given by
  \[
    \Lambda_\alpha(\T)
    =
    \Big\{ f \in \M : \sup_{s > 0} \Big\| s^{1 - \alpha} \, \frac{d P_s}{d \, s}(f) \Big\|_\infty < \infty \Big\}.
  \]
  We will endow $\Lambda_\alpha(\T)$ with the following norm
  \begin{equation}
    \label{eq:HoelderNorm}
    \| f \|_{\Lambda_\alpha}
    =
    \max \bigg\{
      \| f \|_\infty, \,\,
      \sup_{s > 0} \Big\| s^{1 - \alpha} \, \frac{d P_s}{d \, s}(f) \Big\|_\infty \bigg\}.
  \end{equation}
\end{definition}

When no ambiguity can arise, we will denote the space $\Lambda_\alpha(\T)$
just by $\Lambda_\alpha$, removing the dependency on the Markovian semigroup.
It is also worth pointing out that the scale of spaces above can be defined
for $\alpha$ larger that one. Indeed, just take and integer $k > 1$ such that
$0 < \alpha < k$ and define $\Lambda_\alpha$ as
\begin{equation}
  \label{eq:HigherAlpha}
  \Lambda_\alpha(\Ts)
  =
  \Big\{ f \in \M : \sup_{s > 0} \Big\| s^{k - \alpha} \, \frac{d^k P_s}{d \, s^k}(f) \Big\|_\infty < \infty \Big\}.
\end{equation}
The different definitions are comparable for different $k$ as we shall see in Section \ref{sct:Multipliers}. That said, we will mainly work with $0 < \alpha < 1$ since higher-order spaces behave similarly.

When needed, we will endow $\Lambda_\alpha$ with the operator space structure induced by the semigroup $\Ts \otimes \Id = (T_t \otimes \Id)_{t \geq 0}$ acting on $\M \weaktensor M_m(\CC) = M_m[\M]$, whose associated Poisson semigroup is $\Ps \otimes \Id = (P_s \otimes \Id)_{s \geq 0}$. Indeed we can fix their matrix norms as
\[
  \big\| [f_{i, j}] \big\|_{M_m[\Lambda_\alpha]} = \max \bigg\{ \big\| [f_{i, j}] \big\|_{M_m[\M]}, \,\,  \sup_{s > 0} \, \bigg\| s^{1 - \alpha} \, \Big[ \frac{d P_s}{d \, s}(f_{i, j}) \Big] \bigg\|_{M_m[\M]} \bigg\}.
\]

We will also denote by $\| f \|_{\Lambda_\alpha^\circ}$ the \emph{homogeneous H\"older seminorm} given by the second term in \eqref{eq:HoelderNorm}, i.e.
\[
  \| f \|_{\Lambda^\circ_\alpha} = \sup_{s > 0} \Big\| s^{1 - \alpha} \, \frac{d P_s}{d \, s}(f) \Big\|_\infty.
\]
The details of the construction of the homogeneous H{\"o}lder space would be given bellow. Now we will start by noticing that $\Lambda_\alpha$ is a Banach space (indeed, a dual one). For that end we will use the subordination formula \eqref{eq:Subordination}.

\begin{proposition}
  \label{prp:Completeness}
  The spaces $\Lambda_\alpha(\Ts)$ are complete with respect to the norm \eqref{eq:HoelderNorm}
\end{proposition}

\begin{proof}
  Let $(f_m)_m$ be a Cauchy sequence for $\Lambda_\alpha$. In particular it is Cauchy for the operator norm of $\M$. Therefore there is a limit $f$ in the norm of $\M$. For every $s > 0$ take $m_0$ sufficiently large, so that $\| f - f_m \|_\infty \leq s^\alpha$ for every $m > m_0$. To see that $x \in \Lambda_\alpha$, notice that
  \begin{eqnarray*}
    \Big\| s \frac{d}{d \, s} P_s(f) \Big\|_\infty
      & \leq & \Big\| s \frac{d}{d \, s} P_s(f - f_m) \Big\|_\infty + \Big\| s \frac{d}{d \, s} P_s(f_m) \Big\|_\infty\\
      & \lesssim & \| f - f_m \|_\infty + \Big\| s \, \frac{d}{d \, s} P_s(f_m) \Big\|_\infty
      \leq s^\alpha + \Big( \sup_{m > m_0} \| f_m \|_{\Lambda_\alpha} \Big) s^\alpha.
  \end{eqnarray*}
  To see that the convergence happens in the $\Lambda_\alpha$-norm notice fix $\epsilon > 0$ and take $k$ such that $\| f - f_m \|_\infty < s^{\alpha} \, \epsilon / 2$ and $\| f_k - f_m \|_{\Lambda_\alpha} < \epsilon / 2$, for every $m \geq k$, then 
  \begin{eqnarray*}
    \Big\| s^{1 - \alpha} \, \frac{d P_s}{d \, s}(f - f_k) \Big\|_\infty
      & \leq & \Big\| s^{1 - \alpha} \, \frac{d P_s}{d \, s}(f - f_m) \Big\|_\infty
               + \Big\| s^{1 - \alpha} \, \frac{d P_s}{d \, s}(f_m - f_k) \Big\|_\infty\\
      & \leq & \frac{\epsilon}{2} + \frac{\epsilon}{2} = \epsilon. 
  \end{eqnarray*}   
  Since $\epsilon > 0$ is arbitrary we can conclude.
\end{proof}

To the end of proving the duality, let us introduce the unbounded map $\Phi_\alpha: \M \to L^\infty(\RR_+; \M)$ given by
\[
  \Phi_\alpha(f)(s) = s^{1 - \alpha} \, \frac{d P_s}{d \, s}(f).
\]
It is clear that $\Lambda_\alpha$ is a subset of the domain $\dom( \Phi_\alpha ) \subset \M$. 

\begin{proposition}
  \label{prp:Duality}
  Let $\Lambda_\alpha$ and $\Phi_\alpha$ be as above. We have that
  \begin{enumerate}[leftmargin=1cm, label={\rm (\roman*)}, ref={\rm (\roman*)}]
    \item \label{itm:Duality.1}
    As an unbounded operator, $\Phi_\alpha$ is weak-$\ast$ closable.
    \item \label{itm:Duality.2} 
    Its weak-$\ast$ closed domain is isometrically isomorphic to
    $\Lambda_\alpha$.
  \end{enumerate}
  As a consequence of \ref{itm:Duality.2} $\Lambda_\alpha$ is a dual space,
  whose predual is given by the quotient of $L^1(\RR_+; \M_\ast)$ by the
  pre-annihilator of the graph of $\Phi_\alpha$.
\end{proposition}

\begin{proof}
  The fact that $\Lambda_\alpha$ is the domain of $\Phi_\alpha$ is obtained almost by definition. Therefore we are going to prove only \ref{itm:Duality.2}. Take $(F_\lambda)_\lambda \subset \M \oplus L^\infty(\RR_+; \M)$ be a weak-$\ast$ convergent sequence inside the graph of $\Phi_\alpha$. By the Krein-Smulian theorem we can take $F_\lambda$ uniformly bounded without loss of generality. Denote by $F$ the weak-$\ast$ limit of $(F_\lambda)$. By definition we have that
  \[
    F_\lambda = f_\lambda \oplus \Phi_\alpha(f_\lambda) \xrightarrow{ \, \wast \, } f \oplus g_s = F.
  \]
  Notice that for each $s$, we have that
  \[
    g_s= \wstlim_\lambda \Phi_\alpha(f_\lambda)(s)
    = \wstlim_\lambda \frac1{s^\alpha} s \frac{d P_s}{d s}(f_\lambda)
    = \frac1{s^\alpha} s \frac{d P_s}{d s} \Big( \wstlim_\lambda f_\lambda \Big) = \Phi_\alpha(f).
  \]
  We have used the weak-$\ast$ continuity of \ref{itm:Cor.Subordination.2} in Corollary \ref{cor:Subordination}.
  As an application of the dominated convergence theorem we get that if $s \mapsto g_s$ is an operator valued function which converge weak-$\ast$ in $\M$ for every $s$ and it is uniformly bounded in $s$, then $g_s$ converge in the weak-$\ast$ topology of $L^\infty(\RR_+;\M)$. Notice that, by assumption, we have that all the $g_s$ are uniformly bounded and thus we can conclude. The last part follows from the fact that $\Lambda_\alpha$ is isometrically isomorphic to the graph of $\Phi_\alpha$ which is a weak-$\ast$ closed subset of $\M \oplus L^\infty(\RR_+; \M)$.
\end{proof}

Now, we shall introduce the homogeneous $\alpha$-H\"older classes $\Lambda_\alpha^\circ(\Ts)$.
The first technicality that we have to notice is that $\| \cdot \|_{\Lambda^\circ_\alpha}$ is a seminorm that
vanishes on the subspace of $f$ satisfying that $s \mapsto P_s(f)$ is constant. We can characterize that subspace as the kernel of $A^\frac12$. This observation gives that $\ker(A^\frac12) \subset \dom(A^\frac12) \subset \M$ is a weak-$\ast$ closed subspace, see \cite{davies1980OnePSemigroups}. In order to turn $\| \cdot \|_{\Lambda^\circ_\alpha}$ into a norm, the natural thing to do would be to take the quotient of $\M$ by $\ker(A^\frac12)$ and complete with respect to the homogeneous H\"older norm. But then, the second technicality arising is that, even in the case of classical metric spaces, it may not be possible to approximate $f$ in the norm topology of $\Lambda_\alpha^\circ$ by bounded functions. Indeed, as we will see in Section \ref{sec:Examples}, in $\RR^n$ with the usual heat semigroup we have that
\[
  \| f \|_{\Lambda_\alpha^\circ} \sim_{(n,\alpha)} \sup_{x \neq y} \bigg\{ \frac{|f(x) - f(y)|}{|x- y|^\alpha} \bigg\}.
\]
For an unbounded function like $f(x) = |x|^\alpha$ its restrictions $f_N(x) = f(x) \, \1_{\{f(x) \leq N\}}$ do not approximate $f$ in the seminorm above. To circumvent this issue we will define $\Lambda_\alpha^\circ$ as a weak closure.
Let $\M_\alpha$ be the subset of $\M$ given by
\[
  \M_\alpha = \{ f \in \M : \| f \|_{\Lambda^\circ_\alpha} < \infty \}.
\]
And notice that the map $\Phi_\alpha: \M_\alpha \to L^\infty(\RR_+;\M)$ passes to the quotient $\M_\alpha/\ker(A^\frac12)$. We are going to denote that map again by $\Phi_\alpha: \M_\alpha/\ker(A^\frac12) \to L^\infty(\RR_+;\M)$ since no confusion can arise.
\begin{definition}
  \label{def:Homogeneous}
  We define $\Lambda_\alpha^\circ(\Ts)$ as the closure of $\M_\alpha/\ker(A^\frac12)$ with respect to the pullback of the weak-$\ast$ topology induced by $\Phi_\alpha$, i.e. the coarsest topology making the maps
  \[
    \big[ f + \ker(A^\frac12) \big] \longmapsto \varphi \circ \Phi_\alpha(f) \quad \mbox{ for every } \quad \varphi \in \M_\ast 
  \]
  continuous.
\end{definition}

We will end this set of basic definitions recalling that it is possible to substitute the norm of $\M$ by that of $\M^\circ = \M/\ker(A^\frac12)$. Recall that since $\ker(A^\frac12)$ is weak-$\ast$ closed, $\M^\circ$ is a dual space whose natural weak-$\ast$ topology is given by evaluation against the elements of $\M_\ast \subset \M^\ast$ that are null over $\ker(A^\frac12)$. Let us denote by $q: \M \to \M^\circ$ the natural quotient map. It is immediate that there is weak-$\ast$ continuous semigroup $\tilde{P}_s: \M^\circ \to \M^\circ$ such that
\[
  q \left( P_s f \right) = \tilde{P}_s \, q(f)
\]
and similarly for $T_t$. The main advantage of $\M^\circ$ is that over that space $P_s(f)$ converge to $0$ in the weak-$\ast$ topology as $s \to \infty$. When no confusion can arise, we will denote $\tilde{P}_s$ by $P_s$ and omit the quotient map $q$.

\begin{lemma}
  \label{lem:Bubble}
  Let $\Ts$, $\M$ and $\M^\circ$ be as above. Then
  \[
    \left\| s^{1 - \alpha} \, \frac{d P_s}{d \, s}(f) \right\|_{\M} = \left\| s^{1 - \alpha} \, \frac{d P_s}{d \, s}(f) \right\|_{\M^\circ}
  \]
\end{lemma}

\section{Boundedness of the Riesz transforms \label{sct:Riesz}}

\subsection*{The gradient form}
In this section we are going to prove the bounds for the Riesz transforms. Let us start be recalling some of the basic facts and definitions concerning the construction of the gradient form $\Gamma$ as well as the $\Gamma^2 \geq 0$ condition for semigroups.

Since we are going to deal with unbounded quadratic forms that take values in the operators affiliated to $\M$ it will be useful to assume that there is a $\ast$-algebra $\S \subset \M$, satisfying that
\begin{enumerate}[leftmargin=1cm, label={\rm (\Roman*)}, ref={\rm (\Roman*)}]
  \item \label{itm:Test1} $\S$ is weak-$\ast$ dense in $\M$.
  \item \label{itm:Test2} $\S \subset L^1(\M) \cap \M$.
  \item \label{itm:Test3} $A^\frac12[\S] \subset \S$.
  \item \label{itm:Test4} $T_t[\S] \subset \S$ ans $P_s[\S] \subset \S$, for every $0 < s, t$.
\end{enumerate}
Whenever $\S$ satisfies the hypotheses \ref{itm:Test1}-\ref{itm:Test4} we will say that it is a \emph{test algebra}. Observe also that if $\S$ were to be a complete space with respect to some locally convex metric, then $T_s[\S] \subset \S$ would imply that $P_s[\S] \subset \S$ by subordination. But, since we are not assuming that $\S$ has any topology, we should include these extra conditions.

The existence of a test algebra $\S$ can be established in all of the examples we are interested in. Nevertheless, the existence of $\S$ is not a neutral assumption, see Remark \ref{rmk:ClassicalGamma}.

Let $B: \S \times \S \to \S$ be a sesquilinear form. We are going to assume that it is conjugate linear in the first variable and linear in the second. We shall also denote the sesquilinear form by $B(f,g)$ but its associated quadratic form by $B[f] = B(f,g)$. We define its \emph{derived form $B'$} as
\[
  2 \, B'(f,g) = B(Af,g) + B(f, Ag) - A B(f,g),
\]
The interest of the derived form arises from the following observation
\begin{proposition}
  \label{prp:KSBilinear}
  Let $\M$, $\Ts = (T_t)_{t \geq 0} = (e^{-t A})_{t \geq 0}$ and $B: \S \times \S \to \S$
  be as above. The following are equivalent
  \begin{enumerate}[leftmargin=1cm, label={\rm (\roman*)}, ref={\rm (\roman*)}]
    \item \label{itm:KSBilinear1} $B[T_t f] \leq T_t B[f]$.
    \item \label{itm:KSBilinear2} $B' \geq 0$, meaning that $B'[f] \geq 0$ for every $f \in \S$.
  \end{enumerate}
\end{proposition}

The proof is trivial by observing that, for every, $f \in \S$ and $s > 0$, the function $F: [0,s] \to \M$ given by
$F_t =  T_{t} B [ T_{s - t} f]$ is increasing iff \ref{itm:KSBilinear1} holds. But taking a derivative of $F_t$ with respect to $t$ gives that
\begin{equation}
  \label{eq:DerivBilinear}
  \frac{d}{d \, t} F_t = 2 \, T_t B' \big[T_{s - t} f\big].
\end{equation}
Therefore, if $B' \geq 0$, the above quantity is positive. Notice that the calculation is justified by the fact that $A[\S] \subset \S$ and therefore $t \mapsto T_s f$ is weakly differentiable. The other implication follows similarly.

A particularly interesting case is given by the sesquilinear form $B(f,g) = f^\ast g$, whose derived form is denoted by $\Gamma$. In that situation \ref{itm:KSBilinear1} is just the Kadison-Schwartz inequality \eqref{eq:Kadison}, which holds true since we are assuming that $T_s$ is a completely positive map. Therefore, $\Gamma \geq 0$. We can iterate the construction above, denoting $\Gamma$ by $\Gamma^1$ and defining recursively
\[
  2 \, \Gamma^{k+1}(f,g) = \Gamma^{k}(Af,g) + \Gamma^{k}(f, Ag) - A \Gamma^{k}(f,g).
\]
Contrary to the case of $k = 1$, the higher order forms $\Gamma^k$ are not necessarily positive. Unsurprisingly, we will say that a Markovian semigroup $\Ts = (T_{t})_{t \geq 0}$ over $\M$ satisfies the $\Gamma^2 \geq 0$ property whenever $\Gamma^2[f] \geq 0$, for every $f \in \S$.

We will need continuity property of $\Gamma$ in many situations. In order to achieve that we are going to show that $\Gamma$ can be obtained as the (operator valued) Hilbert product of a certain unbounded operator into a Hilbert $W^\ast$-module. The construction is just sketched here and the interested reader can look up the details in \cite{CiSauv2003Derivations, Cipriani2008Book} as well as in \cite[Lemma 1.2.1]{JunMe2010}.


Let us start by noticing that the quadratic form $\Gamma_{(m)}$ defined over the $m \times m$-matrices over $\S$ by
\begin{equation*}
  \Gamma_{(m)} \Big( [f_{i \, j}], [g_{i \, j}] \Big)
  =
  \bigg[ \sum_{k = 1}^m \Gamma \big( f_{k i}, g_{k j} \big) \bigg]_{i,j}
\end{equation*}
coincides with the associated gradient form for the semigroup $\Id \otimes \Ts = (\Id \otimes T_t)_{t \geq 0}$ and therefore it is positive. We can define the following positive $\M$-valued inner product over $\M \algtensor \S$ as 
\[
  \bigg\langle \sum_{j = 1}^N f_j \otimes g_j, \sum_{j = 1}^N f_j \otimes g_j \bigg\rangle_\Gamma
  =
  \sum_{j = 1}^N \sum_{k = 1}^N f_j^\ast \, \Gamma(g_j, g_k) \, f_k.
\]
The fact that $\langle \cdot, \cdot \rangle_\Gamma$ is positive is a consequence of the fact that all the forms $\Gamma_{(m)}$ satisfy that $\Gamma_{(m)}[f] \geq 0$. We will denote by $\M \otimes_\Gamma \M$ the $W^\ast$-Hilbert modules completion of $\M \algtensor \S$ with respect to the weak-$\ast$ topology generated by the seminorms
\[
  p_\varphi(\xi) = \varphi \big( \langle \xi, \xi \rangle_\Gamma^\frac12 \big)
  \quad 
  \mbox{ for every }
  \quad
  \varphi \in \M_\ast.
\]
The map $w: \S \subset \M \to \M \otimes_\Gamma \M$ given by $\delta(f) = \1 \otimes f$ is an unbounded, densely defined operator such that
\[
  \big\langle \delta(f), \delta(g) \big\rangle_\Gamma = \Gamma(f,g).
\]
There are alternative construction that allow us to express $\Gamma$ as the square of a derivation. That can be done by defining the Hilbert $W^\ast$-module as the weak-$\ast$ closure of the
\[
  \Big\{ \sum_{j} f_j \otimes g_j : \sum_j f_j \, g_j = 0 \Big\} \subset \M \algtensor \S
\]
with $\delta(f) = f \otimes \1 - \1 \otimes f$. Recall that every Hilbert $W^\ast$-module can be embedded into $\M[H^c]$ for some Hilbert space $H$. Therefore, we can assume without loss of generality that $\delta$ goes from $\S$ to $\M[H^c]$.

The following result follows from the Schwarz inequality for $\langle \cdot, \cdot \rangle_\Gamma$ and has been already used before in the literature, see \cite[Lemma 3.1]{JunMei2012BMO}.

\begin{lemma}
  \label{lem:Quadratic}
  Let $\mu$ be a finite measure and $f_\omega$ a weakly measurable function, then
    \[
      \Gamma \left[ \int_\Omega f_\omega \, d \mu(\omega) \right] \leq \| \mu \|_{M(\Omega)} \Big( \int_\Omega \Gamma[f_\omega] \, d\mu(\omega) \Big)
    \]
\end{lemma}

We will also point our that in all of the examples that we will cover in this section, the linear map giving a factorization of $\Gamma$ through a Hilbert $W^\ast$ module have a explicit expression and it is not necessary to use the abstract factorization above. Furthermore, it holds that
\begin{equation}
  \label{dia:Commuting}
  \xymatrix@C=7em{
    \M \ar[d]^{T_t} \ar[r]^{\delta} & \M[H^c] \ar[d]^{(T_t \otimes \Id)}\\
    \M \ar[r]^{\delta} & \M[H^c]
  }
\end{equation}
and of course the same intertwining identity follows for the associated Poisson semigroup.

\begin{remark}
  \label{rmk:ClassicalGamma}
    The construction of the gradient that we have presented here omit several technicalities that can occur. For instance, in the theory of Dirichlet forms the standard procedure to construct the gradient $\Gamma$ is to start with a Borel space $(X,\mu)$ and the unbounded form quadratic form $\Ee(f,g) = \langle f, A g\rangle$, which is an object called a \emph{Dirichlet form}. Then, it holds under mild regularity assumptions that $\D = \dom_2(A^\frac12) \cap L^\infty(X,\mu)$ is an algebra and that a Leibniz-type identity is satisfied. The gradient form is defined weakly, for every $\varphi \in C_c(X)$, we take
    \[
      \Gamma_\varphi(f,g)
      =
      \int_\Omega \varphi \, \overline{A f} g + \varphi \, f \overline{A g} - A \varphi \, \overline{f} g d \mu 
    \]
    Regularity results allow to see $\Gamma_\varphi(f,g)$ as a bounded operator acting on $C_c(X)$ and so,
    by Riesz theorem, there exists a signed Borel measure $\Gamma(f,g)$ such that
    $\Gamma_\varphi(f,g) = \langle \Gamma(f,g), \varphi \rangle$. Now, we can see the gradient as an unbounded bilinear form $\Gamma: \D \times \D \to M(X)$, but it may not be the case that $\Gamma(f,g) \ll \mu$. For instance, that holds in the case of certain fractals, see \cite{koskela2012GeoI, koskela2014GeoII}.
\end{remark}

The following proposition amounts to a trivial calculation with the subordination formula of \ref{lem:Subordination} as well as \ref{lem:Quadratic} and thus we omit its proof.
\begin{lemma}
  Let $\Ts$ be a Markovian semigroup satisfying the $\Gamma^2 \geq 0$ property, then
  \[
    \Gamma \big[P_s f\big] \leq P_s \Gamma[f],
  \]
  where $\Ps = (P_s)_{s \geq 0}$ is the associated Poisson semigroup, for every $f \in \M$.
\end{lemma}

Let us finish this discussion on the gradient form by recalling the definition of the \emph{space-time gradient}. Over the functions of $L^\infty(\RR_+; \M) = L^\infty(\RR_+) \weaktensor \M$ we can define the space-time generator as the unbounded operator
\[
  \widehat{A}(f_s)
  =
  A(f_s) - \frac{d^2 }{d \, s^s} f_s
  =
  \bigg( \Id \otimes A + \Big(-\frac{d^2 }{d \, s^s} \Big) \otimes \Id \bigg)(f_s).
\]
This operator is defined over the tensor product of the domains of $\dom(A)$ and $\dom(\partial^2_{ss})$ and its closure is given by taking the closure of the tensor product of the graphs. Since the semigroup generated by $A$ is completely positive, we have that the Markovian \emph{space-time semigroup} $\widehat{\Ts} = (\widehat{T}_t)_{t \geq 0}$ given by $e^{- t \widehat{A}}$ is again positive. We will also denote by $\widehat{\Gamma}$ the space-time gradient form given by
\[
  2 \, \widehat{\Gamma}(f_s,g_s)
  =
  \widehat{A}f_s^\ast \, g_s + f_s^\ast \, \widehat{A}g_s - \widehat{A} (f_s^\ast g_s)
  =
  \Gamma(f_s,g_s) + \bigg( \frac{d}{d \, s} f_s \bigg)^\ast \bigg(\frac{d}{d \, s} g_s \bigg).
\]

The following lemma is immediate.

\begin{lemma}
  \label{lem:SpaceTimeGrad}
  Let $\Ts$ be a Markovian semigroup, $\widehat{\Ts}$ its associated space-time gradient and $\widehat{\Gamma}$ its associated space-time gradient. If $\Ts$ satisfies the $\Gamma^2 \geq 0$, then $\widehat{\Gamma}^2 \geq 0$. 
\end{lemma}

\subsection*{Proof of the boundedness}
The next proposition would be key in order to prove our result.

\begin{proposition}
  \label{prp:DomGamma}
  Let $\Ts = (T_t)_{t \geq 0}$ be a Markovian semigroup and $\Ps$ is associated Poisson semigroup with the $\Gamma^2 \geq 0$ property. For every $2 \leq p \leq \infty$ we have that 
  \[
    \big\| \Gamma\big[P_s f\big]^\frac12 \big\|_p \lesssim_{(p)} \frac1{ s } \, \| f \|_p.
  \]
\end{proposition}

Before going into the proof let us estate the following technical lemma, which has already being used in \cite{JunMe2010}
and \cite[Lemma 1.1]{JunMei2012BMO}.
\begin{lemma}
  \label{lem:GradPos}
  Let $\Ts = (T_t)_{t \geq 0}$ be a Markovian semigroup satisfying $\Gamma^2 \geq 0$ and $\Ps = (P_s)_{s \geq 0}$ its subordinated semigroup. Define the operator-valued function
  \[
    G_s = P_s \left( \frac{d P_s f}{d \, s}^\ast P_s f \right) + P_s \left( P_s f^\ast \frac{d P_s f}{d \, s} \right) - \frac{d P_s}{d \, s} \big( P_s f^\ast P_s f \big)
  \]
  satisfies that
  \[
    -\frac{d}{d \, s} G_s = 2 \, P_s \widehat{\Gamma}\big[P_s f \big]
  \]
\end{lemma}

\begin{proof}
  The proof amount to a simple calculation, indeed
  \begin{eqnarray*}
    -\frac{d}{d \, s} G_s
      & = & - \frac{d^2 P_s}{d \,  s^2} \big( |P_s f|^2 \big) - \frac{d P_s}{d \, s} \left( \frac{d P_s f}{d \, s}^\ast P_s f \right) - \frac{d P_s}{d \, s} \left( P_s f^\ast \frac{d P_s f}{d \, s} \right) \\
      &   & \quad + \, \frac{d P_s}{d \, s} \left( P_s f^\ast \frac{d P_s f}{d \, s} \right) + P_s \left( \Big| \frac{d P_s f}{d \, s} \Big|^2 \right) + P_s \left( P_s f^\ast \frac{d^2 P_s f}{d \, s^2} \right) \\
      &  & \quad + \, \frac{d P_s}{d \, s} \left( \frac{d P_s f}{d \, s}^\ast P_s f \right)+ P_s \left( \frac{d^2 P_s f}{d \, s^2}^\ast P_s f \right) + P_s \left( \Big| \frac{d P_s f}{d \, s} \Big|^2 \right) \\
      & = & P_s \left( \frac{d^2 P_s f}{d \, s^2}^\ast P_s f \right) + P_s \left( P_s f^\ast \frac{d^2 P_s f}{d \, s^2} \right) - \frac{d^2 P_s}{d \,  s^2} \big( |P_s f|^2 \big) + 2 P_s \left( \Big| \frac{d P_s f}{d \, s} \Big|^2 \right)\\
      & = & 2 P_s \, \Gamma[P_s f] + 2 P_s \left( \Big| \frac{d P_s f}{d \, s} \Big|^2 \right)
      \quad = \quad 2 \, P_s \widehat{\Gamma}[P_s f],
  \end{eqnarray*}
  we have used the fact that $\partial_s P_s = - A^\frac12 \, P_s$ in the first equation and
  the definition of $\Gamma$ in the third. That concludes the proof.
\end{proof}

\begin{proof}{{\bf (of Proposition \ref{prp:DomGamma})}}
  The proof is heavily based in the subordination formula \ref{lem:Subordination}.
  Start by noticing that, by the $\Gamma^2 \geq 0$ property, we have that
  \begin{equation}
    \label{eq:DomGamma1}
    \Gamma[ P_s f ]
    \leq P_{\frac{s}{2}} \, \Gamma \big[ P_{\frac{s}{2}} f \big]
    \leq P_{\frac{s}{2}} \, \Gamma \big[ P_{\frac{s}{2}} f \big] + P_{\frac{s}{2}} \Big| \frac{d}{d \, s} P_{\frac{s}{2}}(f) \Big|^2
    = P_{\frac{s}{2}} \, \widehat{\Gamma} \big[ P_{\frac{s}{2}} f \big].
  \end{equation}
  Now, using the change of variable $s \mapsto 2 s$, which adds just a constant to the calculations, we obtain that
  \begin{equation}
    \big\| \Gamma[P_{2 s} f]^\frac12 \big\|_p^2 = \big\| \Gamma[P_{2 s} f] \big\|_{\frac{p}{2}}
    \leq \big\| P_s \, \widehat{\Gamma}[P_s f] \big\|_{\frac{p}{2}}.
  \end{equation}
  By application of the Lemma \ref{lem:GradPos} it follows that
  \begin{eqnarray*}
    \big\| P_s \widehat{\Gamma}[P_s f] \big\|_{\frac{p}{2}}
      & = & \left\| \frac12 \, \frac{d}{d \, s} G_s \right\|_{\frac{p}{2}}\\
      & = & \left\| \frac12 \, \frac{d}{d \, s} \left\{ P_s \left( \frac{d P_s f}{d \, s}^\ast P_s f \right) + P_s \left( P_s f^\ast \frac{d P_s f}{d \, s} \right) - \frac{d P_s}{d \, s} \big( P_s f^\ast P_s f \big) \right\}  \right\|_{\frac{p}{2}}\\
      & = & \left\| \frac12 \, \frac{d}{d \, s} \left\{ \int_0^\infty \int_0^\infty \int_0^\infty \Psi_s(u,v,w) \, T_u \big( T_v f^\ast T_w f \big) \, d u \, d v \, d w \right\}  \right\|_{\frac{p}{2}},
  \end{eqnarray*}
  where $\Psi_s$ is given by the $\phi_s$ appearing in the subordination formula
  \[
    \Psi_s(u,v,w) = \frac{d}{d \, s} \phi_s(u) \phi_s(v) \phi_s(w) + \phi_s(u) \frac{d}{d \, s} \phi_s(v) \phi_s(w) - \phi_s(u) \phi_s(v) \frac{d}{d \, s} \phi_s(w).
  \]
  Using Lemma \ref{lem:Subordination}\ref{itm:Subordination.2} we obtain that $s^2 |\partial_s \Psi_s|$ is a finite measure, therefore
  \begin{eqnarray*}
    \big\| P_s \, \widehat{\Gamma}[P_s f] \big\|_{\frac{p}{2}}
      & \leq & \frac{1}{s^2} \, \int_0^\infty \int_0^\infty \int_0^\infty s^2 \, \bigg| \frac{d}{d \, s} \Psi_s(u,v,w) \bigg| \, \big\| T_u \big( T_v f^\ast \,  T_w f \big) \big\|_{\frac{p}{2}} \, d u \, d v \, d w \\
      & \lesssim & \frac{1}{s^2} \left( \int_0^\infty \int_0^\infty \int_0^\infty s^2 \,  \bigg| \frac{d}{d \, s} \Psi_s(u,v,w) \bigg| \, d u \, d v \, d w\right) \| f \|^2_{p} \quad \lesssim \quad \frac{1}{s^2} \,  \| f \|_p^2.
  \end{eqnarray*}
  Taking square roots gives the desired result.
\end{proof}
Notice that the formula of Proposition \ref{prp:DomGamma} holds for the larger form $\widehat{\Gamma}$.

\begin{remark}
  Another elementary but noteworthy observation is that if the expression $\Gamma[P_s f]$
  is changed by the derivative $\partial_s P_s f$, then the fact that its norm in a certain Banach
  space grows like $s^{-1}$ when $s \to 0^+$ is equivalent to the existence of an analytic
  extension of the semigroup $t \mapsto P_s$ over a sector $\Sigma_\theta$ of the complex plane
  \begin{equation}
    \label{eq:sector}
    \Sigma_\theta = \{ z \in \CC : |\arg(z)| < \theta \mbox{ and } \Re(z) > 0 \} \subset \CC.
  \end{equation}
  Every Markovian semigroup is analytic over $L^p(\M)$ for $1 < p < \infty$ by a classical argument
  of complex interpolation, see \cite[Chapter 2]{Ste1970}. The analyticity when $p = \infty$ or $1$
  does not always hold. Nevertheless, it is always the case that the subordinated Poisson semigroup
  is analytic over $\M$ by the subordination formula. Therefore, Proposition \ref{prp:DomGamma}
  can be understood as an improvement over analyticity in which the time derivative is changed
  by the gradient. 
\end{remark}

With Proposition \ref{prp:DomGamma} at hand we can proceed to prove the main theorem.

\begin{theorem}
  \label{thm:RieszBound}
  Let $\Ts$ be a Markovian semigroup and $\Ps$ is subordinated Poisson semigroup
  If $\Gamma^2 \geq 0$, we have that
  \begin{equation}
    \sup_{s > 0} \left\{ s^{1 - \alpha} \, \big\| \widehat{\Gamma}[P_s f] \big\|^\frac12_{\M^\circ} \right\}
    \sim_{(\alpha)}
    \sup_{s > 0} \left\{ s^{1 - \alpha} \, \left\| \frac{d}{d \, s} P_s f \right\|^\frac12_{\M^\circ} \right\}
  \end{equation}
\end{theorem}

\begin{proof}
  Using the previous lemma, we obtain that 
  \begin{eqnarray*}
    \left\| \Gamma \left[ \frac{d P_s}{d \, s}(f) \right]^\frac12 \right\|_\infty
      & = & \left\| \, \Gamma \left[ P_{\frac{s}{2}} \frac{d P_{\frac{s}{2}}}{d \, s}(f) \right]^\frac12 \right\|_\infty\\
      & \lesssim & \frac1{s} \left\| \frac{d P_{\frac{s}{2}}}{d \, s}(f) \right\|_\infty
      \quad  \lesssim \quad \frac1{s^{2 - \alpha}} \, \| f \|_{\Lambda_\alpha^\circ}
  \end{eqnarray*}
  The right hand side is larger that the quantity that we obtain when we change the norm of $\M$ by that of $\M_\circ$.
  Now, notice that for any $f \in \M_\circ$, it holds that
  \[
    -P_s(f) = \int_s^\infty \frac{d P_t}{d \, t}(f) d\,t = \wstlim_{N \to \infty} \int_s^N \frac{d P_t}{d \, t}(f) d\,t.
  \]
  Therefore, by Lemma \ref{lem:Quadratic}
  \begin{eqnarray*}
    \big\| \, \Gamma[ P_s f]^\frac12 \big\|_{\M^\circ}
      & =    & \left\| \, \Gamma \left[ \int_s^\infty \frac{d P_t}{d \, t}(f) \, d \, t \right]^\frac12 \right\|_{\M^\circ}\\
      & \leq &  \int_s^\infty \left\| \, \Gamma \left[ \frac{d P_t}{d \, t}(f) \right]^\frac12 \right\|_{\M^\circ} \, d\, t \\
      & \lesssim & \int_s^\infty \frac{1}{t^{2 - \alpha}} \| f \|_{\Lambda_\alpha^\circ} \, d\, t \quad = \quad \frac1{1 - \alpha} \frac1{s^{1 - \alpha}}\| f \|_{\Lambda_\alpha^\circ}.
  \end{eqnarray*}
  and with that we conclude.
\end{proof}

\begin{remark}
  Notice that in the theorem above we only need $\Gamma^2 \geq 0$ for one of the directions.
  The other is trivial and holds by definition of $\widehat{\Gamma}$.
  Observe also that we can replace the left hand side with $\Gamma$ and retain an inequality stating that
  \begin{equation}
    \label{eq:OnesidedRiesz}
    \sup_{s > 0} \left\{ s^{1 - \alpha} \, \big\| \Gamma[P_s f] \big\|^\frac12_{\M^\circ} \right\}
    \lesssim_{(\alpha)}
    \sup_{s > 0} \left\{ s^{1 - \alpha} \, \left\| \frac{d}{d \, s} P_s f \right\|^\frac12_{\M^\circ} \right\}
  \end{equation}
  and, since $\partial_s P_s = A^\frac12 P_s$, we can interpret the inequality above as a form of
  boundedness of an abstract Riesz transform, i.e. an operator exchanging $\Gamma$ by $A^\frac12$.
  To make that intuition precise let us define a Hilbert valued analogue of $\Lambda_\alpha^\circ$.
  First fix a Hilbert space $H$ and notice that $\M[H^c]$ is given isometrically isomorphic
  to $(\M \weaktensor \B(H))(\1 \otimes p)$, where $p$ is any rank one projection on $H$.
  The semigroup $P_s$ extends to $P_s \otimes \Id: \M \weaktensor \B(H) \to \M \weaktensor \B(H)$
  and this semigroup commutes with the right multiplication by $\1 \otimes p$ and thus it is well
  defined over $\M[H^c]$. We can define $\Lambda_\alpha^\circ(\Ts;H^c)$ as
  \[
    \Lambda_\alpha^\circ(\Ts;H^c) = \Lambda_\alpha^\circ(\Ts \otimes \Id) \, (\1 \otimes p),
  \]
  where $\Lambda_\alpha^\circ(\Ts \otimes \Id)$ is the homogeneous $\alpha$-H\"older classes
  over $\M \weaktensor \B(H)$.
  Now, when $\Gamma[f] = \langle \delta(f), \delta(f) \rangle$ and $\delta$ satisfies the intertwining \eqref{dia:Commuting} we can define the Riesz transform as the Hilbert-valued operator
  $\R = \delta \, A^\frac12$. Taking $f = A^\frac12 A^{-\frac12} f$, we have that
  \begin{eqnarray*}
    s^{1-\alpha} \, \left\| \Gamma \left[ \frac{d P_s}{d \, s}(A^{-\frac12} f) \right] \right\|_{\M^\circ}^\frac12
      & = & s^{1-\alpha} \, \left\| \left\langle \delta \, \frac{d P_s}{d \, s} (A^{-\frac12} f), \delta \, \frac{d P_s}{d \, s} (A^{-\frac12} h)\right\rangle \right\|_{\M^\circ}^\frac12\\
      & = & s^{1-\alpha} \, \left\| \left\langle \frac{d P_s}{d \, s} \delta (A^{-\frac12} f),  \frac{d P_s}{d \, s} \delta (A^{-\frac12} h)\right\rangle \right\|_{\M^\circ}^\frac12\\
      & = & s^{1-\alpha} \, \left\| \left\langle \frac{d P_s}{d \, s} \R(f), \frac{d P_s}{d \, s} \R(f) \right\rangle \right\|_{\M^\circ}^\frac12\\
      & = & s^{1-\alpha} \, \left\| \frac{d P_s}{d \, s} \R(f)  \right\|_{\M^\circ[H^c]} \leq \| \R(f) \|_{\Lambda_\alpha^\circ(H^c)}.
  \end{eqnarray*}
  Taking supremum in $s > 0$ makes both quantities above the same, therefore we can interpret Theorem \ref{thm:RieszBound} as 
  \begin{equation}
    \label{eq:RieszWastModule}
    \big\| \R: \Lambda_\alpha^\circ(\Ts) \to \Lambda_\alpha^\circ(\Ts;H^c) \big\| < \infty,
  \end{equation}
  and this gives Theorem \ref{thm:Riesz}.
\end{remark}

\section{Examples \label{sct:Examples}}
\label{sec:Examples}

The estimates of the last section can be used to produce explicit characterizations of the H{\"o}lder classes in several natural settings.

\subsection*{Riemannian manifolds} 
Let $(M,g)$ be a complete Riemannian manifold. From this point on we shall denote its volume form by $d \, \mbox{vol}$ a. In a Riemannian manifold it is possible to define a gradient operator $\nabla: C_c^\infty(M) \to \mathfrak{X}^\infty_c(T M)$ and a divergence $\mbox{div}: \mathfrak{X}^\infty_c(TM) \to C_c^\infty(M)$. where $\mathfrak{X}^\infty_c(TM)$ is the space of smooth and compactly supported vector fields. The Laplace-Beltrami operator is given by $\Delta = - \mbox{div} \, \nabla$ and it generates a Markovian semigroup of operators $T_t = e^{- t \Delta}$, which is the classical heat semigroup over the manifold $M$, the interested reader can look up more of the background in \cite[Section 1.8/Chapter 9]{Berger2012Panoramic}. A classical regularity argument easily gives that $T_t$ is given by a continuous integral kernel $h_t(x,y)$ satisfying that
\[
  T_t f (x) = \int_M h_t(x,y) \, f(y) \, d\vol(y),
\]
similarly there is an integral kernel $p_s$ for the associated subordinated Poisson semigroup.

Recall that in this example we can take $C_c^\infty(M)$, the algebra of smooth and compactly supported functions as our test algebra. After a calculation using the so-called Bochner formulas, it was shown by P-A. Meyer \cite{Meyer1976}, see also \cite{Ledoux2000}, that
\begin{eqnarray*}
  \Gamma(f,h) & = & g \big(\nabla f, \nabla h \big),\\
  \Gamma^2(f,h) & = & \big\langle \nabla^2 f, \nabla^2 h \big\rangle_{\mathrm{HS}} + \Ric(d f, d h), 
\end{eqnarray*}
where $\nabla^2$ is the second covariant derivative and $\langle \cdot , \cdot \rangle_{\mathrm{HS}}$ denotes the Hilbert-Schmidt product over the matrices on $T_x M$ induced by $g_x$. Meyer also showed that the $\Gamma^2 \geq 0$ condition holds iff $\Ric \geq 0$.

Recall also that, if two points $x, y \in M$ are in the same connected component their distance can be defined as the infimum over the length of all curves joining $x$ and $y$. We can take the distance to be infinite if the points are in different connected components. 

We are going to recall the following standard definitions

\begin{definition}
  \label{def:PoincareDoubling}
  Let $(M,g)$ be a complete Riemannian manifold, we will say that
  \begin{enumerate}[leftmargin=1cm, label={\rm (\roman*)}, ref={\rm (\roman*)}]
    \item \label{itm:Poincare} $M$ satisfies a \emph{scale-invariant Poincare inequality} iff
    for every $0 < r < \infty$ and $x \in M$
    \begin{equation}
      \label{eq:Poincare}
      \tag{Po}
      \bigg( \int_{B_x(r)} |f - f_{B_x(r)}|^2 \, d \vol \bigg)^\frac12 \lesssim r \bigg( \int_{B_x(2 r)} |\nabla f|^2 d \vol \bigg)^\frac12,
    \end{equation}
    where $f_B$ is the average of $f$ over $B_x(r)$,
    \item $M$ is \emph{doubling} iff
    \begin{equation}
      \label{eq:Doubling}
      \tag{Dou}
      \vol(B_x(2 \, r)) \leq D_0 \, \vol(B_x(r))
    \end{equation}
    for a constant $D_0$ independent of $r$ and $x$.
  \end{enumerate}
\end{definition}

We are going to use the following theorem.

\begin{theorem}{{\bf (\cite[Theorem 5.4.12]{Saloff2002})}}
  \label{thm:GaussianBound}
  Let $(M,g)$ be a complete Riemannian manifold satisfying \eqref{eq:Doubling} and \eqref{eq:Poincare}. It holds that
  \[
    \frac1{\vol(B_x(\sqrt{t}))} e^{-\beta_0 \, \frac{d(x,y)^2}{t}}
    \,\, \lesssim \,\,
    h_t(x,y)
    \,\, \lesssim \,\,
    \frac1{\vol(B_x(\sqrt{t}))} e^{-\beta_1 \, \frac{d(x,y)^2}{t}}
  \]
\end{theorem}

We will also need the following lemma

\begin{lemma}
  \label{lem:LipschitzSub}
  \label{lem:GuassianDoubling} 
  Let $\phi_s$ be the function on the subordination formula \eqref{eq:Subordination}
  and let $\Phi: \RR \to \RR$ be an increasing and right-continuous doubling function
  i.e: satisfying that $\phi(2 \, r) \leq D_0 \, \Phi(r)$
  \begin{enumerate}[leftmargin=1cm, label={\rm (\roman*)}, ref={\rm (\roman*)}]
    \item \label{itm:LipschitzSub}
    It holds that 
    \[
      \int_0^\infty \Big|s \frac{\partial}{\partial s} \phi_s(v) \Big| \, v^\frac{\alpha}{2} \, d v \lesssim_{(\alpha)} s^\alpha.
    \]
    \item \label{itm:GuassianDoubling}
    Let $d m_\Phi$ be the Lebesgue-Stieltjes derivative of $\Phi$. We have that
    \[
      \int_0^\infty e^{- \beta \frac{r^2}{u}} \, r^\alpha \, d m_\Phi(r)
      \lesssim_{(\beta,\alpha)}
      u^\frac{\alpha}{2} \, \Phi(\sqrt{u})
    \]
  \end{enumerate}  
\end{lemma}

\begin{proof}
  The first point follows after a straightforward calculation. For the second, use that, by definition
  \begin{eqnarray*}
    \int_0^\infty \bigg| \frac{d}{d \, r} \Big\{ e^{-\beta \frac{r^2}{u}} r^{\alpha} \Big\} \bigg| \, \Phi(r) \, d r
      & \leq & \int_0^\infty \bigg( \frac{2 r}{u} e^{-\beta \frac{r^2}{u}}
        + \alpha r^{\alpha - 1} e^{-\beta \frac{r^2}{u}} \bigg) \Phi(r) \, d r\\
      & =    & u^\frac{\alpha}{2} \, \int_0^\infty \bigg( 2 r e^{-\beta r^2}
        + \alpha r^{\alpha - 1} e^{-\beta r^2} \bigg) \Phi(r\sqrt{u}) \, d r\\
      & \leq & u^\frac{\alpha}{2} \Phi(\sqrt{u})
        \bigg( \sum_{k = 0}^\infty D_0^k \, \int_{2^k-1}^{2^{k+1}-1} \big( 2 r e^{-\beta r^2}
        + \alpha r^{\alpha - 1} e^{-\beta r^2} \big) \, d r \bigg).
  \end{eqnarray*}
  But the series on the last expression are summable.
\end{proof}

Let us denote by $\| \cdot \|_\alpha$ the classical H\"older seminorm over functions in $M$ given by
\[
  \| f \|_\alpha = \sup_{x \neq y} \left\{ \frac{|f(x) - f(y)|}{d(x,y)^\alpha} \right\}.
\]

\begin{theorem}
  \label{thm:HoelderEqMani}
  \
  \begin{enumerate}[leftmargin=1cm, label={\rm (\roman*)}, ref={\rm (\roman*)}]
    \item \label{itm:HoelderMani.1} If $M$ is doubling and satisfies the Poincare inequality \eqref{eq:Poincare}, then
    it holds that $\displaystyle{\| f \|_{\Lambda_\alpha^\circ(\T)} \lesssim \| f \|_\alpha}$.
    \item \label{itm:HoelderMani.2} If $\Ric \geq 0$ and $M$ is connected, then we have that 
    $\displaystyle{\| f \|_\alpha \lesssim \| f \|_{\Lambda_\alpha^\circ(\T)}}$.
  \end{enumerate}
\end{theorem}

\begin{proof}
  We will prove first \ref{itm:HoelderMani.1}. Fix a function $f$ with $\| f \|_\alpha < \infty$. We have, by the definition of the integral kernel of $P_s$, that
  \[
    s \frac{d P_s f}{d \, s}(x) = \int_M s \frac{\partial}{\partial \, s} p_s(x,y) \, f(y) \, d \vol(y). 
  \]
  Now, using that $\partial_s p_s(x,y)$ has $0$-integral over $y$ for every $x \in M$, we obtain that
  \[
    s \frac{d P_s f}{d \, s}(x)
    =
    \int_M \frac{\partial}{\partial \, s} p_s(x,y) \big( f(x) - f(y) \big) \, d \vol(y)
  \]
  and so
  \begin{eqnarray}
    \left| s \, \frac{d P_s f}{d \, s}(x) \right|
      & \leq & \int_M \left| \frac{\partial}{\partial \, s} p_s(x,y) \right| \, \big| f(x) - f(y) \big| \, d \vol(y) \nonumber\\
      & \leq & \| f \|_\alpha \int_M \left| \frac{\partial}{\partial \, s} p_s(x,y) \right| \, d(x,y)^\alpha \, d \vol(y) \label{eq:defHoel.1}\\
      & =    & \| f \|_\alpha \int_M \left| \int_0^\infty \frac{\partial}{\partial \, s} \phi_s(u) \, h_u(x,y) \, d u \right| \, d(x,y)^\alpha \, d \vol(y) \label{eq:defHoel.2}\\
      & \leq & \| f \|_\alpha \int_M \int_0^\infty \left| \frac{\partial}{\partial \, s} \phi_s(u) \right| \, h_u(x,y) \, d u \, d(x,y)^\alpha \, d \vol(y) \nonumber \\
      & =    & \| f \|_\alpha \int_0^\infty \left| \frac{\partial}{\partial \, s} \phi_s(u) \right| \underbrace{\bigg( \int_M h_u(x,y) \, d(x,y)^\alpha \, d \vol(y) \bigg) }_{\mathrm{(I)}} d u. \label{eq:defHoel.3}
  \end{eqnarray}
  We have used the definition of $\| f \|_\alpha$ in \eqref{eq:defHoel.1}, the subordination formula in \eqref{eq:defHoel.2} and Fubini's theorem in \eqref{eq:defHoel.3}. We can estimate the integral $\mathrm{(I)}$ as follows. Denote by $\Phi_x(r)$ the doubling function given by  $\Phi_x(r)= \vol(B_x(r))$. Taking a change of variable by calling $r$ to $d(x,y)$, we get that
   \begin{eqnarray*}
     \mathrm{(I)}
       & =    & \int_M h_u(x,y) \, d(x,y)^\alpha \, d \vol(y) \\
       & \leq & \frac1{\Phi_x(\sqrt{u})} \int_M e^{- \beta_1 \, \frac{d(x,y)^2}{u}} \, d(x,y)^\alpha \, d \vol(y) \,  =  \, \frac1{\Phi_x(\sqrt{u})} \int_0^\infty e^{- \beta_1 \, \frac{r^2}{u}} \, r^\alpha \, d m_{\Phi_x}(r)
   \end{eqnarray*}
   and using the Lemma \ref{lem:GuassianDoubling} we get that $\mathrm{(I)} \lesssim u^\alpha$. Now, using the Lemma \ref{lem:LipschitzSub} we get
   \[
     \left| s \, \frac{d P_s f}{d \, s}(x) \right| \, \lesssim \, \| f \|_\alpha \, \int_0^\infty \left| \frac{\partial}{\partial \, s} \phi_s(u) \right| u^\frac{\alpha}{2} \, d u \, \lesssim \,  s^\alpha \| f \|_\alpha 
   \]
   and taking supremum over al $x$ we obtain \ref{itm:HoelderMani.1}.
   
   For \ref{itm:HoelderMani.2} fix $f \in \Lambda_\alpha^\circ(\Ts)$ and $s > 0$, to be determined later. We have that
   \[
     \big| f(x) - f(y) \big| \leq \big| f(x) - P_s f(x) \big| + \big| P_s f(x) - P_s f(y) \big| + \big| P_s f(y) - f(y) \big| = \mathrm{(I)} + \mathrm{(II)} + \mathrm{(III)}.
   \]
   The terms $\mathrm{(I)}$ and $\mathrm{(III)}$ are estimated similarly
   \[
     \mathrm{(I)} \leq \left| \int_0^s \frac{d P_t f(x)}{d \, t} \, d t \right| 
     \lesssim_{(\alpha)}
     \| f \|_{\Lambda^\circ_\alpha(\Ts)} s^{\alpha}.
   \]
   While, for $\mathrm{(II)}$ we are going to use the inequality \eqref{eq:OnesidedRiesz} obtained in Theorem \ref{thm:RieszBound} and the fact that $\Gamma[f] = |\nabla f|^2$. take a continuous path $\gamma$ from $x$ to $y$ of unit speed and left $\ell = d(x,y) + \epsilon$, we have that
   \begin{eqnarray}
     \mathrm{(II)}
      & = &  \left| \int_0^\ell g \big( \nabla P_s f (\gamma(t)), \dot{\gamma}(t) \big) \, d t \right|\\
      & \leq & \int_0^\ell \left| \nabla P_s f (\gamma(t)) \right| \, d t \, \leq   \, \ell \, \left\| \nabla P_s f \right\|  \, = \, \| f \|_{\Lambda_\alpha^\circ} \, (d(x,y) + \epsilon) \frac1{s^{1 - \alpha}}
   \end{eqnarray}
   But now, choosing $s = d(x,y)$ and noting that $\epsilon > 0$ can be made arbitrarily small  gives the result.
\end{proof}

There is certain overlapping between the hypotheses of both Theorem \ref{thm:HoelderEqMani} and Corollary \ref{cor:HoelderEqMani}. For instance $\Ric \geq 0$ plus geodesic completeness implies a scale invariant Poincare inequality in $L^1$. Also $\Ric \geq$ implies doublingness, see \cite{LottWeak2005}. Using this extra information we obtain the following corollary.

\begin{corollary}
  \label{cor:HoelderEqMani}
  In the case of a complete and connected Riemannian manifold $(M,g)$ with $\Ric \geq 0$ we have that $\Lambda^\circ_\alpha(\Ts)$ and $\Lambda^\circ_\alpha(M)$ are quasi-isometric. The same holds for $\Lambda_\alpha(\Ts)$ and $\Lambda_\alpha(M)$.
\end{corollary}

\subsection*{Quantum torii and transference}

Let $\Theta = (\Theta_{i j})_{i j}$ be a real antisymmetric $n\times n$-matrix and $\Theta_{\downarrow}$ its lower triangular part. We can define a bicharacter $\chi: \ZZ^n \times \ZZ^n \to \CC$ given by
\[
  \chi(k_1, k_2) = e^{2 \pi i \left\langle k_1, \Theta_{\downarrow} k_2 \right\rangle},
\]
where $k_1, k_2 \in \ZZ^n$. Now, we can treat $\chi$ as a $2$-cocycle and define the ($n$-dimensional) quantum torus $\A_\Theta$ as $\CC \rtimes_{\chi} \ZZ^n$ as introduced in \cite{zellerMeier}. More concretely, we have that $\A_\Theta$ is the von Neumann algebra generated by the unitaries $u_k$ of $\B(\ell^2 \ZZ^n)$
\[
  u_k f (\ell) = \chi( k, -\ell ) \, f(\ell - k).
\]
Those operators satisfy the following commutation relation
\[
  u_{k_1} u_{k_2}
  = e^{2 \pi i \left\langle k_1, \Theta_{\downarrow} k_2 \right\rangle} u_{k_1 + k_2}
  = e^{2 \pi i \left\langle k_1, \Theta k_2 \right\rangle} u_{k_2} u_{k_1}
\]
Observe that in $\A_\Theta$ is still true that $\tau(f) = \langle \delta_0, f \delta_0 \rangle$ is still a faithful normal trace, in particular $\A_\Theta$ is a finite von Neumann algebra. Every element of $\A_\Theta$ can be understood as a sum
\[
  f = \sum_{k \in \ZZ^n} a_k \, u_k
\]
that can be understood in the weak-$\ast$ sense. Notice that in the case of $\Theta = 0$ the $u_k$ become translation operators or, in the other side of the Fourier transform, characters. To further the analogy with the Fourier transform we will denote $a_k = \tau( f \, u_k^\ast)$ by $\widehat{f}(k)$. 

It is worth noting that the translation action of $\TT^n$ in $L^\infty(\TT^n)$ has a natural analogue in the context of quantum torii. Let us define the $\ast$-homomorphism $\sigma: \A_\Theta \to L^\infty(\TT^n) \weaktensor \A_\Theta$ by linear extension of $\sigma(u_k) = \exp_k \otimes u_k$, where $\exp_k \in L^\infty(\TT^n)$ is the function given by $\exp_k(\theta) = e^{2 \pi i \langle \theta, k \rangle}$. A routine application of the Fell absorption principle gives that $\sigma$ is indeed a normal $\ast$-homomorphism $\sigma: \A_\Theta \to L^\infty(\TT^n) \weaktensor \A_\Theta$, see see \cite{GonJunPar2017singular} for more on the details. Evaluation in the first component gives an weak-$\ast$ continuous action $\sigma: \TT^n \to \Aut(\A_\Theta)$ that is given by
\[
  \sigma_{z} f
  =
  \sigma_{z} \bigg( \sum_{k \in \ZZ^n} \widehat{f}(k) \, u_k \bigg)
  =
  \sum_{k \in \ZZ^n} e^{2 \pi i \langle z, k \rangle} \widehat{f}(k) \, u_k.
\]
This action allows us to extend the heat semigroup. Indeed, let $\Delta$ be the unbounded operator given by $\Delta(u_k) = 4 \pi^2 \, |k|^2 \, u_k$. It generates a Markovian semigroup $(T_t)_{t \geq 0}$ such that
\[
  T_t f = \int_{\TT^n} h_t(z) \sigma_z(f) \, d z,
\]
where $h_t$ is the convolution kernel of the heat semigroup on $\TT^n$. Similarly. we can define the associated Poisson semigroup generated by $\Delta^\frac12$, which in turn is given by 
\[
  P_s f = \int_{\TT^n} p_s(z) \sigma_z(f) \, d z,
\]
where $p_s$ is the associated Poisson convolution kernel on $\TT^n$. A scale of spaces $\Lambda_\alpha(\A_\Theta)$ was introduced by Weaver in \cite{Weaver1998alphaLip} as the subspaces of $\A_\Theta$ such that
\[
  \sup_{t \in \RR} \left\| \frac{\sigma_{t e_i}(f) - f}{|t|^\alpha} \right\|_\infty < \infty
\]
for every $i \in \{1, 2, ..., n\}$ with a norm (equivalent) to 
\[
  \| f \|_{\alpha} 
  = \max \left\{ \| f \|_\infty, \, \sup_{z \in \TT^n} \left\| \frac{\sigma_z(f) - f}{|z|^\alpha} \right\|_\infty \right\}
\]
We want to prove that the H\"older classes defined in terms of the semigroup $\Ps = (P_s)_{s \geq 0}$ and those given by the quantity above coincide.
\begin{theorem}
  \label{thm:ExTorii}
  Let $\Theta$, $\A_\Theta$, $\sigma$ and $\Ps $ be as above. We have that
  \[
    \| f \|_{\Lambda_\alpha(\Ts)} \sim \| f \|_{\alpha}
  \]
  and the spaces $\Lambda_\alpha(\Ts)$ and $\Lip_\alpha(\A_\Theta)$ are quasi isometric.
\end{theorem}

The proof is completely trivial once we notice that $\Gamma$ has an explicit expression as
\[
  \Gamma(f,g) = \big\langle \nabla f, \nabla g \big\rangle_{\A_\Theta}
\]
where the inner product is the $\A_\Theta$-valued inner product of the Hilbert $W^\ast$ module $\A_\Theta[\ell^2_c]$, where $\ell^2$ is an $n$-dimensional Hilbert space. The gradient is given by
\[
  \nabla f
  = \nabla \bigg( \sum_{k \in \ZZ^n} \widehat{f}(k) \, u_k \bigg)
  = \sum_{k \in \ZZ^n} \sum_{j = 1}^n \widehat{f}(k) \, u_k \otimes 2 \pi i \, k_j e_j,
\] 
where $k = (k_1, k_2, ..., k_n) \in \ZZ^n$ and $e_i$ is a base for $\ell^2$. Noticing that the semigroup satisfies the $\Gamma^2 \geq 0$ property trivially allows us to apply Theorem \ref{thm:RieszBound}.

\begin{proof}[Proof. (of Theorem \ref{thm:ExTorii})]
  The proof is analogous to that of Theorem \ref{thm:HoelderEqMani}, therefore we are only going to sketch it.
  \begin{eqnarray*}
    \left\| s \, \frac{d P_s}{d \, s}(f) \right\|_\infty
      & =    & \left\| \int_{\TT^n} s \frac{\partial p_s}{\partial s}(z) \, \sigma_z(f) \, d z \right\|_\infty\\
      & =    & \left\| \int_{\TT^n} s \frac{\partial p_s}{\partial s}(z) \, \big( \sigma_z(f) - f \big) \, d z \right\|_\infty\\
      & \leq & \int_{\TT^n} \Big| s \frac{\partial p_s}{\partial s}(z) \Big| \, \big\| \sigma_z(f) - f \big \|_\infty \, d z \\
      & \leq & \| f \|_{\alpha} \, \int_{\TT^n} \Big| s \frac{\partial p_s}{\partial s}(z) \Big| \, |z|^\alpha \, d z  \lesssim_{(n,\alpha)} \| f \|_{\alpha} s^\alpha. 
  \end{eqnarray*}
  For the other inequality, again we make
  \[
    \begin{split}
      \big\| \sigma_z(f) & - f \big\|_\infty\\
        & \leq \left\| \sigma_z(f) - P_s \sigma_z(f) \right\|_\infty + \left\| \sigma_z(P_s f) - P_s f \right\|_\infty + \left\| f - P_s f \sigma_z(f) \right\|_\infty\\
        & = \mathrm{(I)} + \mathrm{(II)} + \mathrm{(III)}.
    \end{split}
  \]
  The terms $\mathrm{(I)}$, $\mathrm{(III)}$ are estimated like in the proof of Theorem \ref{thm:HoelderEqMani} by taking $s = |z|$. For the  middle term we have that
  \[
    \sigma_z(P_s f) - P_s f = \int_0^{|z|} \frac{d}{d \, t} \sigma_{z |z|^{-1} t} (P_s f) \, d t
    = \int_0^{|z|} \big\langle (\nabla P_s f), t z|z|^{-1} \big\rangle \, d t
  \]
  and applying Theorem \ref{thm:RieszBound} we can conclude.
\end{proof}

\begin{remark}
  The technique presented here has applicability beyond the case of quantum torii. Let $\M$ is a semifinite von Neumann algebra with a family of automorphisms $\sigma: G \to \Aut(\M,\tau)$, where $G$ is a Lie group. It is possible to extend notions from $G$ to $\M$. For instance if $X$ is a right-invariant vector field in $T_e(G)$ ---the Lie algebra of $G$--- and $\gamma_t = e^{t X}$ is the one-parameter subgroup of $G$ generated by $X$ we can define the operator $X_\sigma$ acting on a dense subset of $\M$ as
  \[
    X_\sigma(f) = \frac{d}{d \, t} {\Big|}_{t = 0} \sigma_{\gamma_t}(f).
  \]
  Similarly, it is possible to transfer to $\M$ the action of an invariant Laplacian acting on $G$ and of the associated Heat semigroup. This allows to define semigroup H\"older classes. In the case of groups with positive Ricci curvature and actions satisfying the Fubini property
  \[
    \1 \tau(f) = \int_G \sigma_g(f) \, d\mu(g)
   \]
   it is possible to extend the result above and prove that characterize the space of operators $f$ such that $\sigma_g(f) - f \in O(d(e,g)^\alpha)$ in terms of semigroups.
   
   One interesting example that we shall cover in more detail in a forthcoming paper \cite{Gon2019CZHoelder} is that of quantum Euclidean spaces $\R_\Theta$. They are defined as deformations of $\RR^n$ pretty much in the same way as quantum torii are. Let $\Theta$ and $\Theta_{\downarrow}$ be as before and define the bicharacter $\chi: \RR^n \times \RR^n \to \CC$ given by 
  \[
    \chi(\xi_1, \xi_2) = e^{2 \pi i \left\langle \xi_1, \Theta_{\downarrow} \xi_2 \right\rangle}.
  \]
  Then, we can define $\R_\Theta \subset \B(L^2 \RR^n)$ as the twisted crossed product
  $\CC \rtimes_\chi \RR^n$, where $\chi$ is treated as a $2$-cocycle. Like in the quantum
  torii case there is a normal $\ast$-homomorphism given by
  $\sigma(u_\xi) = \exp_\xi \otimes u_\xi$ which extends to
  $\sigma: \R_\Theta \to L^\infty(\RR^n) \weaktensor \R_\Theta$. In this case all the
  hypotheses above hold and the $\alpha$-H\"older classes defined in terms of the
  transfered Poisson semigroup equal those defined in terms of the automorphisms
  $\sigma_z(f) = \sigma(f)(z)$. The interested reader can look \cite{GonJunPar2017singular} for more details.
\end{remark}

\subsection*{Group von Neumann algebras}
The last example of this article will be given by group algebras. Let $G$ be a locally compact and second countable group and let $\lambda: G \to U(L^2 G)$ be the left regular representation given by $\lambda_g \xi (h) = \xi(g^{-1} h)$ for $\xi \in L^2(G)$. The reduced group algebra $\L G$ is given by 
\[
  \L G = \wstspan\{ \lambda_g \}_{g \in G}.
\]
Furthermore, the set of elements in $\L G$ that can be expressed as an operator-valued integral of the form
\[
  \lambda(F) = \int_G F(g) \lambda_g \, d \mu(g) 
\]
where $F \in L^1(G)$ is weak-$\ast$ dense in $\L G$. In the case in which $G$ is unimodular, ie admits a right and left invariant measure, we can extend the following operator
\[
  \tau \bigg( \int_G F(g) \lambda_g \, d \mu(g)  \bigg) = F(e)
\]
a priori defined for $F \in C_c(G) \ast C_c(G)$ to the whole von Neumann algebra, see \cite[Chapter 8]{Ped1979}. In the non-unimodular case the formula above defines just a faithful normal weight. In both cases it satisfies the Plancherel identity that gives that $\lambda$ extends to an isometric isometry $\lambda: L^2(G) \to L^2(\L G, \tau)$, where the last space can be understood as the Gel'fand-Neumark-Segal construction associated to $\tau$.

Let $m \in L^\infty(G)$. We say that an operator $T_m: L^2(\L G) \to L^2(\L G)$ is a \emph{Fourier multiplier} iff it is given by $T_m(\lambda(f)) = \lambda(m \, f)$. The function $m$ is called the \emph{symbol} of $T_m$. We are interested in studying the Markovian semigroup of $\L G$. Luckily, they have a well known characterization. Recall that a function $\psi: G \to \RR_+$is said to be \emph{conditionally negative} iff $\psi(e) = 0$ and for every finite subset $\{ g_1, ..., g_r \} \subset G$ and vector 
$(v_1, ..,v_r) \in \C^n$  we have
\[
  \sum_{i=1}^r v_i = 0
  \quad \Longrightarrow \quad
  \sum_{i=1}^{r} \sum_{j=1}^n \bar{v}_i \psi \big( g_{i}^{-1} g_j \big) v_j \leq 0.
\]
Similarly, we will say that $\psi$ is \emph{symmetric} iff $\psi(g) = \psi(g^{-1})$.
Let $H$ be a real Hilbert space and $\pi: G \to O(H)$ be an orthogonal representation. We will say
that $\beta: G \to H$ is a $1$-cocycle with respect to $\pi$ iff
$\beta(g \, h) = \beta(g) + \pi(g) \beta(h)$. Observe that this is the same as saying that the map $\xi \mapsto \pi(g) \xi + \beta(g)$ is an affine isometric representation of $G$. The following result links those concept together, see \cite[Appendix C]{BeHarVal2008} or \cite[Chapter 1]{CheCowJoJulgVal2001}.

\begin{theorem}
  \label{thm:genCN}
  Let $\Ts=(T_t)_{t \geq 1}$ be a semigroup of Fourier multipliers.
  given by $T_t(\lambda_g) = e^{-t \psi(g)} \lambda_g$ for some
  $\psi: G \to \RR_+$ and the following statements 
  are equivalent:
  \begin{enumerate}[leftmargin=1cm, label={\rm (\roman*)}, ref={\rm (\roman*)}]
    \item $\Ts$ is a Markovian semigroup.
    \item $\psi : G \to \RR_+$  is conditionally negative
    \item There is a $1$-cocycle $\beta:G \to H$, such that $\psi(g) = \| \beta(g) \|_{H}^2$
  \end{enumerate}
\end{theorem}

As a consequence of the result above we have an explicit characterization of the gradient form $\Gamma$ associated to $\psi$. Let $\delta$ be the unbounded operator $\delta: \L G \to \L G[H^c]$ given by extension of $\delta(\lambda_g) = \lambda_g \otimes \beta(g)$. 
It is possible to see that such operator is weak-$\ast$ closable and that it satisfies both that $\Gamma(f,g) = \langle \delta(f), \delta(g) \rangle$ and the commutation relation \eqref{dia:Commuting}.  We will use the this when obtaining characterizations of boundedness for Fourier multipliers over $\Lambda_\alpha^\circ(\L G)$ in Section \ref{sct:Multipliers}.

\section{Morrey inequalities and ultracontractivity \label{sct:Ultrcontractivity}}

In this section we are going to prove that an analogue of the Morrey inequality, see \cite[4.27]{Adams2003SobolevBook} or \cite[Chapter 1]{Saloff2002}, formulated strictly in terms of semigroups, is equivalent to the ultracontractivity property. This can be seen as an extension of the work of Varopoulos on semigroups and Hardy-Littlewood-Sobolev inequalities, see \cite{Va1985} or \cite[Chapter II]{VaSaCou1992}. The proof is completely elementary

\begin{definition}{{\bf (\cite[p. 9]{VaSaCou1992})}}
  \label{def:Ultracontractivity}
  Let $\Ts = (T_s)_{t \geq 0}$ be a semigroup. It is sad to satisfy the ultracontractivity
  property with respect to $0 < n$ iff there exists $1 \leq p < q \leq \infty$
  \begin{equation}
    \label{eq:Ultracontractivity}
    \tag{$R_n^{p,q}$}
    \big\| T_t: L_p(\M) \to L_q(\M) \big\| \, \lesssim \, t^{- \frac{n}{2} \left\{ \frac1{p} - \frac1{q} \right\}}
  \end{equation}
\end{definition}

\newcommand{\Rn}{\hyperref[eq:Ultracontractivity]{($R_n$)} }

It is well known that if \eqref{eq:Ultracontractivity} holds for some $p < q$ then it holds for all of them, see \cite[11.2.2 Proposition]{VaSaCou1992}. Therefore, we will denote such property simply by \Rn. We will also use that if $\Ts$ has the \Rn property, then its associated Poisson semigroup $\Ps$ satisfies \hyperref[eq:Ultracontractivity]{($R_{2 n}$)}. The proof is a straightforward calculation involving the subordination formula, similar to the proof of Lemma \ref{lem:LipschitzSub}. 

\begin{proof}[{\bf Proof. (of Theorem \ref{thm:Morrey})}]
  First we are going to prove \ref{itm:Morrey1} implies \ref{itm:Morrey2}. Fixing $\alpha = 1 - n p^{-1}$, we have that
  \begin{eqnarray*}
    \| f \|_{\Lambda_\alpha^\circ} = \sup_{s > 0} \left\{ s^{1 - \alpha} \, \bigg\| \frac{d P_s}{d \, s}(f) \bigg\|_\infty \right\}
      & =        & \sup_{s > 0} \left\{ s^{1 - \alpha} \, \big\| P_s A^\frac12 f \big\|_\infty \right\}\\
      & \lesssim & \sup_{s > 0} \left\{ s^{1 - \alpha - \frac{n}{p}} \, \big\| A^\frac12 f \big\|_p \right\} \quad  = \quad \big\| A^\frac12 f \big\|_p.    
  \end{eqnarray*}
  For the other direction the calculation is very similar
  \begin{eqnarray*}
    \big\| P_s f \big\|_\infty
      & =        & \big\| P_s A^\frac12 A^{-\frac12} f \big\|_\infty\\
      & =        & \bigg\| \frac{d P_s}{d \, s} A^{-\frac12} f \bigg\|_\infty\\
      & =        & s^{\alpha - 1} \, \big\| A^{-\frac12} f \big\|_{\Lambda_\alpha^\circ} \, \lesssim \, s^{-\frac{n}{p}} \, \| f \|_{p}
  \end{eqnarray*}
  This proves that $\Ps$ has \hyperref[eq:Ultracontractivity]{($R_{2 n}$)}. But this is equivalent to any of the Sobolev inequalities of $p < n$ and therefore to $\Ts$ satisfying \hyperref[eq:Ultracontractivity]{($R_{n}$)}.
\end{proof}

\begin{remark}
  This gives the Sobolev-type inequalities over the critical range $p > n$ that are equivalent to ultracontractivity. To the knowledge of the author it is still unknown if there is a Sobolev type inequality in terms of semigroups in the critical exponent $p=n$ (something along the lines of Trudinger's inequality) that would allow to recover ultracontractivity. 
\end{remark}

\begin{remark}
  \label{rmk:LocalAndCBR}
  There are several variations of the property \Rn that have appeared in the literature and generally the theorem above can be be easily tweaked to give a Morrey inequality equivalent to each of this variations. We mention two. The first are the localized properties \eqref{eq:Rn0} and \eqref{eq:RnInf} bellow, see also \cite[11.3.4 Theorem]{VaSaCou1992}
  \begin{equation} 
    \label{eq:Rn0}
    \tag{$R_n(0)$}
    \big\| T_t: L_p(\M) \to L_q(\M) \big\|
      \, \lesssim \, t^{- \frac{n}{2} \left\{ \frac1{p} - \frac1{q} \right\}} \quad \mbox{ for } 0 < t \leq 1
  \end{equation} 
  \begin{equation}
    \label{eq:RnInf}
    \tag{$R_n(\infty)$}
    \big\| T_t: L_p(\M) \to L_q(\M) \big\|
      \, \lesssim \, t^{- \frac{n}{2} \left\{ \frac1{p} - \frac1{q} \right\}} \quad \mbox{ for } 1 < t < \infty
  \end{equation}
  Both cases admit characterizations in terms of Sobolev embeddings. In our context the second would be equivalent to the inequality
  \[
    \| f \|_{{\Lambda_{1 - \frac{n}{p}}^\circ}} \lesssim \big\| A^\frac12 f \big\|_p + \| f \|_p
  \]
  for some $p > n$. The other family of variations of the inequality above has appeared in \cite{GonJunPar2015} and is obtained by changing the norm of $\| T_t : L_p(\M) \to L_q(\M) \|$ by its complete analogue, which can incomparably larger when $\M$ is not hyperfinite. In that situation the property \Rn with complete bounds is equivalent to a completely bounded analogue of the Morrey inequality above.
\end{remark}

\section{Multipliers \label{sct:Multipliers}}
In this section we will prove the boundedness between the homogeneous H\"older classes of different families of singular integral operators, namely spectral multipliers and Fourier multipliers. First, we will see the boundedness of the analytic H\"ormander-Mikhlin multipliers introduced by Stein in \cite[p. 3-4]{Ste1970} as $m(A^\frac12)$, where $m: \RR_+ \to \CC$ is a function satisfying that
\begin{equation}
  \label{eq:InfHMCond}
  m(\lambda) = \int_0^\infty \lambda \, e^{-t \, \lambda} M(t) \, dt,
\end{equation}
for some $M \in L^\infty(\RR_+)$. We will denote by $\| m \|_{\mathfrak{M}(\omega)}$ the quantity $\| M \|_\infty$. Observe that $m$ is an analytic function. Indeed, the condition on $m$ imposed by \eqref{eq:InfHMCond} can be understood as an infinite dimensional version of the Marcinkiewicz multiplier condition. The infinite smoothness of the multiplier condition for general semigroups seems necessary since all $\RR^n$, for arbitrary large $n$, are included. The analytic character appears even in the commutative case when studying invariant multipliers over $\SL_2(\RR)$, see \cite[p. 4, note 4]{Ste1970}.

\begin{theorem}
  Let $\Ts = (T_t)_{t \geq 0}$ be a Markovian semigroup and $\Ps = (P_s)_{s \geq 0}$ be its associated Poisson semigroup. We have that
  \[
    \big\| m(A^\frac12): \Lambda_\alpha^\circ(\Ts) \to \Lambda_\alpha^\circ(\Ts) \big\|
    \, \lesssim_{(\alpha)} \,
    \| m \|_{\mathrm{\mathfrak{M}(\omega)}},
  \]
  for every $0 < \alpha < 1$.
\end{theorem}

\begin{proof}
  The proof follows from a straightforward calculation.
  \[
    \begin{array}{>{\displaystyle}r>{\displaystyle}l>{\displaystyle}l}
      \big\| m( A^\frac12 ) f \big\|_{\Lambda_\alpha^\circ} 
        & = & \sup_{s > 0} s^{-\alpha} \, \bigg\| s \, \frac{d P_s}{d  \, s}\big( m(A^\frac12) f \big) \bigg\|_\infty \vspace{2pt}\\
        & = & \sup_{s > 0} s^{-\alpha} \, \bigg\| s \, \frac{d P_s}{d  \, s} \bigg( \int_0^\infty M(t) \frac{d P_t}{d \, t}(f) \, d t \bigg) \bigg\|_\infty \vspace{2pt} \\     
        & = & \sup_{s > 0} s^{-\alpha} \, \bigg\| s \, \int_0^\infty M(t) \frac{d^2 P_{s + t}}{d \, t^2}(f) \, d t \bigg\|_\infty \vspace{3pt} \\
        & \lesssim_{(\alpha)} & \| m \|_{\mathrm{\mathfrak{M}(\omega)}} \, \| f \|_{\Lambda_\alpha^\circ} \, \sup_{s > 0} \bigg\{ s^{1 - \alpha} \, \bigg( \int_s^\infty \frac1{t^{2-\alpha}} \, d t \bigg) \bigg\}\\
        & \lesssim_{(\alpha)} & \| m \|_{\mathrm{\mathfrak{M}(\omega)}} \, \| f \|_{\Lambda_\alpha^\circ}
    \end{array}
  \]
\end{proof}

\subsection*{Group algebras}
The result above is, as we have explained before, infinite-dimensional in nature. Now, we are going to see results with finite smoothness on group algebras.

Let us start with the following lemma
\begin{lemma}
  \label{lem:Eqsquare}
  Let $\Ps = (P_s)_{s \geq 0}$ be as above. For every $0 < \alpha < 1$ we have that
  \[
    \| f \|_{\Lambda_\alpha^\circ} \sim_{(\alpha)} \sup_{s > 0} \bigg\{ s^{-\alpha} \, \Big\| s^2 \, \frac{d^2 P_s}{d \, s^2}(f) \Big\|_\infty \bigg\}
  \]
\end{lemma}

\begin{proof}
  It is trivial, either from the subordination formula \eqref{eq:Subordination} or from the fact that $\Ps = (P_s)_{s \geq 0}$ is an analytical semigroup over $L^1(\M)$, that $s \, \partial_s P_s$ is a uniformly bounded family of operators in $L^1(\M)$ and thus also in $\M$ by the self-adjointness of $P_s$. Therefore
  \[
    \sup_{s > 0} \bigg\{ s^{-\alpha} \,  \Big\| s^2 \, \frac{d^2 P_s}{d \, s^2}(f) \Big\|_\infty \bigg\}
    \lesssim_{(\alpha)}
    \| f \|_{\Lambda_\alpha^\circ}.
  \]
  The other inequality is also easy and follows after a straightforward calculation. Indeed, we have that
  \[
    \begin{array}{>{\displaystyle}rl>{\displaystyle}l}
      s^{1 - \alpha} \, \bigg\| \frac{d P_s}{d \, s}(f) \bigg\|_{\M^\circ}
        & = & s^{1 - \alpha} \, \bigg\| \int_s^\infty \frac{d^2 P_t}{d \, t^2}(f) \, d t \bigg\|_{\M^\circ} \vspace{3pt}\\
        & \leq & s^{1 - \alpha} \, \int_s^\infty \frac1{t^{2 - \alpha}} \, t^{2 - \alpha} \Big\| \frac{d^2 P_t}{d \, t^2}(f) \Big\|_{\M^\circ} \, d t \vspace{3pt} \\
        & \leq & \sup_{t > 0} \bigg\{ t^{2 - \alpha} \Big\| \frac{d^2 P_t}{d \, t^2}(f) \Big\|_{\M^\circ} \bigg\} \, s^{1 - \alpha} \, \int_s^\infty \frac1{t^{2 - \alpha}} \, d t \vspace{2pt} \\
        & \lesssim_{(\alpha)} & \sup_{t > 0} \bigg\{ t^{2 - \alpha} \Big\| \frac{d^2 P_t}{d \, t^2}(f) \Big\|_{\M^\circ} \bigg\}.
    \end{array}
  \]
  Taking supremum in $s$ we can conclude.
\end{proof}

Observe that the lemma can be generalized to higher derivatives, giving the well-definedness of the $\Lambda_\alpha^\circ$-norms for higher $\alpha$ as defined in \eqref{eq:HigherAlpha}. We will denote $z \mapsto z \, e^{-z/2}$ by $\eta(z)$. Then, we have that $\eta(s A^\frac12) \, \eta(s A^\frac12) = s^2 \, \partial_s^2 P_s$, while $\eta(s A^\frac12 ) = s \partial_s P_{s/2}$. It is interesting to point out that the norm equivalence in Lemma \ref{lem:Eqsquare} can be understood as saying that we can change $\eta(s A^\frac12)$ for $\eta^2(s A^\frac12)$ while maintaining the norm up to constants. It is likely that the lemma also holds for general $\rho(s A^\frac12)$, where $\rho$ is any $H^\infty_0$-function, i.e. a bounded holomorphic function over a sector $\Sigma_\theta \subset \CC$, as in \eqref{eq:sector}, that tends to $0$ both at $0$ and $\infty$. That equivalence of norms could be understood as an instance of the $H^\infty$-functional calculus, see \cite{JunMerXu2006}.

Let $G$ be a unimodular locally compact topological group and let $\Delta: \dom(\Delta) \subset L^2(G) \to L^2(G)$ be an unbounded operator generating a Markovian semigroup $\Ss = (S_s)_{s \geq 0}$. Recall that, whenever $\Ss$ is invariant under right translations there is a positive and unbounded operator $\hat{\Delta}$ in $(\L G)^\wedge_+$, the extended positive cone of $\L G$, such that
\begin{equation}
  \label{eq:MultSymbol}
  \lambda(\Delta f) = \hat{\Delta} \, \lambda(f),
\end{equation}
a similar result holds for left invariant operators but with the unbounded operator $\hat\Delta$ acting on the left. We will call the operator $\hat\Delta$ the \emph{multiplication symbol} of $\Delta$. The construction of $\hat{Delta}$ and the proof of \eqref{eq:MultSymbol} are a inmediate in the case of $\Delta$ bounded and straighforward exercise in the case of general unbounded operators. The interested reader can look up the details in \cite{GonJunPar2015}. Recall the following definition from \cite{GonJunPar2015}.

\begin{definition}[{\bf \cite[Definition 3.1]{GonJunPar2015}}]
  \label{def:Cogrowth}
  Let $\Delta: \dom(\Delta) \subset L^2(G) \to L^2(G)$ be the infinitesimal generator of a right-invariant Markovian semigroup and $\hat\Delta \in (\L G)^\wedge_+$ its multiplication symbol. We say that $\Delta$ has $\cogrowth(\hat\Delta) \leq n$ iff
  \[
    (\1 + \hat\Delta)^{- \frac{s}{2}} \in L^1(\L G)
  \]  
  for every $s > n$.
  The critical $n$ satisfying the property above would be denoted by $\cogrowth(\hat\Delta)$. A similar notion can be defined in the case of left-invariant Markovian semigroup generators without further complications.
\end{definition}

Observe that the property above is a form of ``polynomial growth'' over the dual object of $G$, which is described by the algebra $\L G$. We would use the following result from \cite{GonJunPar2015} which related the co-growth of $\hat\Delta$ with the local Sobolev dimension of $\Delta$.

\begin{proposition}[{\bf \cite[Theorem 3.6./Remark 3.7.]{GonJunPar2015}}]
  \label{prp:EqSobCogrowth}
  Let $\Delta: \dom(\Delta) \subset L^2(G) \to L^2(G)$ be the infinitesimal generator of a right-invariant Markovian semigroup of operators $\Ss = (S_t)_{t \geq 0}$ satisfying that $S_t[C_0(G)] \subset C_0(G)$. The following are equivalent
  \begin{enumerate}[leftmargin=1cm, label={\rm (\roman*)}]
    \item $S_t$ has a local ultracontractivity property {\rm ($R_{n+\epsilon}(0)$)}
    of Remark \ref{rmk:LocalAndCBR} for every $\epsilon> 0$, i.e.
      \begin{equation}
        \label{eq:R0local}
        \tag{$R_{(n+\epsilon)}^{p,q}(0)$}
        \big\| S_t: L^p(G) \to L^q(G) \big\| \leq t^{\frac{n + \epsilon}{2} \big\{ \frac1{p} - \frac1{q} \big\}}, 
      \end{equation}
      for every $0 < t \leq 0$.
    \item $\cogrowth(\hat\Delta) = n$
    \item The generator $\Delta$ satisfies that $\forall \epsilon > 0$, we have that
      \[
        \| f \|_{L^\infty(G)} \, \lesssim_{(\epsilon)} \, \big\| (\1 + \Delta)^\frac{s}{2} f \big\|_{L^2(G)},
      \]
      for every $s = \epsilon + n/2$.
  \end{enumerate}   
\end{proposition}

The following theorem can be understood as a generalization to group algebras of the boundedness of classical Marcinkiewicz multipliers over homogeneous H\"older classes, see \cite{SteinZygmund1967}. Recall that in the following theorem we are fixing a locally compact group $G$ and a conditionally negative function $\psi: G \to \RR_+$. As we have said before, we will denote the H\"older classes coming from the Poisson semigroup 
\[
  P_s(\lambda_g) = e^{-s \, \psi(g)^\frac12} \lambda_g
\]
by $\Lambda_\alpha^\circ(\L G)$, removing the dependency on $\psi$. We will also fix $\eta$ as the function $\eta(z) = z e^{-z / 2}$ and a right translation invariant Markovian semigroup $S_t$ acting on $G$ and satisfying the property \hyperref[eq:R0local]{$R^n(0)$} above.

\begin{theorem}
  \label{thm:MarcikiewiczDual}
  Let $m: G \to \CC$ be a bounded measurable function satisfying that
  \[
    \sup_{t > 0} \Big\{ \| m \, \eta \big( t \psi^\frac12 \big) \|_{W^{2,s}_\Delta(G)} \Big\} < \infty,
  \]
  where the norm of $W^{2,s}(G)$ is given by $\| f \|_{W^{2,s}_\Delta(G)} = \| (\1 + \Delta)^{s/2} f \|_2$.
  We have that, for every $0 < \alpha < 1$, it holds that
  \[
    \big\| T_m: \Lambda^\circ_\alpha(\L G) \to \Lambda^\circ_\alpha(\L G) \big\|
    \lesssim_{(\alpha)} \sup_{t > 0} \Big\{ \| m \, \eta \big( t \psi^\frac12 \big) \|_{W^{2,s}_\Delta(G)} \Big\}
  \]
\end{theorem}

In order to prove the theorem we are going to use the following easy lemma.
\begin{lemma}
  \label{lem:TrivialMult}
  Assume that $\cogrowth(\hat{\Delta}) = n$, then we have that for every $\epsilon > 0$ and function $m \in L^2(G)$
  \[
    \| \lambda(m) \|_{L^1(\L G)} \lesssim_{(\epsilon)} \big\| (1 + \Delta)^\frac{n + \epsilon}{4} f \big\|_2,
  \]
  as usual we will interpret the right hand side to be infinite if $f$ is not in the domain of $(1 + \Delta)^{(n+\epsilon)/4}$.
  As a consequence we have that 
  \[
    \big\| T_m: \L G \to \L G \big\| \lesssim \| m \|_{W^{2,s}(G)}
  \]
  where $s > n/2$. 
\end{lemma}

\begin{proof}
  The proof amount to a calculation
  \begin{eqnarray*}
    \tau \big( | \lambda(m) | \big)
      & = & \tau \Big( \big| \lambda( (\1 + \Delta)^{-s/2} \, (\1 + \Delta)^{s/2} m ) \big| \Big)\\
      & = & \tau \Big( \big| (\1 + \hat{\Delta})^{-s/2} \, \lambda((\1 + \Delta)^{s/2} m ) \big| \Big)\\
      & \leq & \big\| (\1 + \hat{\Delta})^{-s/2} \big\|_2 \, \| m \|_{W^{2,s}(G)}\\
      & \lesssim & \| m \|_{W^{2,s}(G)}
  \end{eqnarray*}
  We have used the finite cogrowth in the last inequality.
\end{proof}

\begin{proof}
  The proof is quite straightforward. We will start using the Lemma \ref{lem:Eqsquare} to see that
  \[
    \| T_m(f) \|_{\Lambda_\alpha^\circ(\psi)}
    \sim_{(\alpha)} \sup_{t > 0} \Big\{ t^{-\alpha} \, \big\| \eta^2(t A^\frac12) T_m f \big\|_\infty \Big\}
    \sim_{(\alpha)} \sup_{t > 0} \Big\{ t^{-\alpha} \, \big\| \eta(t A^\frac12) T_m \eta(t A^\frac12) f \big\|_\infty \Big\}.
  \]
  Note that $\eta(t A^\frac12)f = T_{t \eta(\psi)^\frac12} f$, we will denote $m_t$ the function $m_t(g) = \eta(t \psi(g)^\frac12) m(g)$ and notice that $\eta(t A^\frac12) T_m = T_{m_t}$. Now, using Lemma \ref{lem:TrivialMult} we get that
  \[
    \begin{array}{>{\displaystyle}r>{\displaystyle}l>{\displaystyle}l}
      \| T_m(f) \|_{\Lambda_\alpha^\circ}
        & \sim_{(\alpha)} & \sup_{t > 0} \Big\{ t^{-\alpha} \, \big\| T_{m_t}(\eta(t A^\frac12) f) \big\|_\infty \Big\}\\
        & \leq & \sup_{t > 0} \Big\{ \big\| T_{m_t}: \L G \to \L G \big\| \Big\} \, \sup_{t > 0} \Big\{ t^{-\alpha} \, \big\| \eta(t A^\frac12) f \big\|_\infty \Big\}\\
        & \lesssim & \sup_{t > 0} \Big\{ \big\| m \, \eta\big( t \psi^\frac12 \big) \big\|_{W^{2,s}(G)} \Big\} \, \| f \|_{\Lambda_\alpha^\circ}
    \end{array}
  \]
  and that concludes the proof.
\end{proof}

\begin{remark}
  This result is analogous to \cite[Theorem C]{GonJunPar2015} for $L^p$ spaces. The main difference between the conditions of the two being that, in the case of $\Lambda_\alpha^\circ$, since the norm is defined symmetrically, we don't need to control left and right invariant semigroups of operators simultaneously. By the contrary, although the norm of $L^p(\L G)$ is still symmetric under conjugation, the proof of the theorem requires the use of square functions, which in the noncommutative case forces us to control both the column and the row behavior.
\end{remark}

We are going to see a concrete instance of the Theorem \ref{thm:MarcikiewiczDual} in the case in which $G$ is a Lie group and $\Delta$ is a sublaplacian. Recall that $\XX = \{ X_1, X_2, ... X_r \} \subset T_e(G)$, a family of right-invariant vector fields in $G$ satisfies the \emph{H\"ormander condition} whenever its iterated commutators generate the whole Lie algebra. Given any such system we can associate to it  a subriemannian metric given by
\[
  d_\XX(x,y)
  =
  \inf \bigg\{ \Big( \int_0^1 \sum_{j = 1}^r |a_j(t)|^2 \, d t \Big)^\frac12 :
               \dot{\gamma}(t) = \sum_{j = 1}^r a_j(t) \, X_j \bigg\}
\]
where the infimum is taken over all curves with $\gamma(0) = x$ and $\gamma(1) = y$. The H\"ormander condition ensures that the metric $d_\XX$ is finite for every two points in the same connected component of $G$. It was shown early on that the volumes on the Balls of $d_\XX$ for small radius grow like
\begin{equation*}
  \mu \big( B_e(r) \big) \sim r^{D_0(\XX)}, \quad \mbox{ for every } \quad 0 \leq t \leq 1,
\end{equation*} 
where $\mu$ is the Haar measure of $G$ and $D_0(\XX)$ is the local dimension given by
\begin{equation*}
  D_0(\XX) = \sum_{j = 1}^\infty j \dim ( F_j / F_{j - 1} )
\end{equation*}
where $F_0 = \{0\}$ and $F_j$ is the span of $[F_{j - 1}, \XX]$.
We can also construct an invariant sublaplacian operator $\Delta_\XX$ given by
\[
  \Delta_\XX(f) = - \sum_{j = 1}^r X_j X_j (f).
\]
The Markovian semigroup $\Ss = (S_t)_{t \geq 0}$ satisfies that its fixed functions on $L^\infty(G)$ are just the constants. It is also worth noticing that the distance $d_\XX$ can be recovered from $\Delta_\XX$ as the gradient metric associated to $\Ss$. Using techniques involving the parabolic Harnack principle it is possible to show the following.

\begin{theorem}[{{\bf \cite[Theorem VIII 2.9.]{VaSaCou1992}}}]
  \label{thm:BallsHeatGroup}
  Let $G$ be a unimodular Lie group and $\XX \subset T_e(G)$ a H\"ormander system as above. We have that for every $0 \leq \leq t \leq 1$
  \[
    \big\| S_t: L^1(G) \to L^\infty(G) \big\| \lesssim \frac1{\mu\big(B_e(\sqrt{t})\big)} \sim t^{-\frac{D_0(\XX)}{2}}.
  \]
\end{theorem}

Observe that this imply trivially that $\Ss$ satisfies the property \hyperref[eq:R0local]{($R_{D_0(\XX)}(0)$)} and therefore the whole family of ultracontractivity properties \hyperref[eq:R0local]{($R_{\epsilon + D_0(\XX)}(0)$)} and any of the three properties of Proposition \ref{prp:EqSobCogrowth}. As a consequence we get the Theorem \ref{thm:corMarcikiewiczDual}

\section{Campanato-type formulas \label{sct:Campanato}}

In this section we will study alternatives to the norm of $\Lambda_\alpha^\circ$ defined in terms of decaying mean oscillation. In the classical case this approach was pioneered by Campanato in \cite{Campanato1963}. Our abstract semigroup approach is more close to that of \cite{JunMei2012BMO, Mei2012H1BMO, Mei2008Tent}. Let us introduce the following two seminorms over elements of $\M$

\begin{definition}
  \label{def:CampanatoNorms}
   Let $\Ts = (T_t)_{t \geq 0}$ be a Markovian semigroup and $\Ps = (P_s)_{s \geq 0}$ its subordinated Poisson semigroup. For each $0 < \alpha < 1$ and $f \in \M$ we define the following quantities
  \begin{eqnarray*}  
    \| f \|_{\LLip^c_\alpha}
      & = & \sup_{s > 0} \left\{ s^{-\alpha} \, \big\| P_s | f - P_s f|^2 \big\|_\infty^\frac12 \right\} \\
    \| f \|_{\llip^c_\alpha}
      & = & \sup_{s > 0} \left\{ s^{-\alpha} \, \big\| P_s|f|^2 - |P_s f|^2 \big\|_\infty^\frac12 \right\}
  \end{eqnarray*}
\end{definition}

The $c$ in the quantities above is shorthand notation for \emph{column} (as in the column operator space structure). Similarly we can also define a \emph{row} version of the quantities above as
\begin{eqnarray*}
  \| f \|_{\LLip^r_\alpha}
    & = & \| f^\ast \|_{\LLip^c_\alpha}, \\
  \| f \|_{\llip^r_\alpha}
    & = & \| f^\ast \|_{\llip^c_\alpha}.
\end{eqnarray*}
Observe that we have just changed the square $f \mapsto |f|^2 = f^\ast f$ by $f f^\ast$. We will also define the symmetric seminorms as follows
\begin{eqnarray*}
  \| f \|_{\LLip_\alpha}
    & = & \max \big\{ \| f \|_{\LLip^c_\alpha}, \| f \|_{\LLip^r_\alpha}\big\} \\
  \| f \|_{\llip_\alpha}
    & = & \max \big\{ \| f \|_{\llip^c_\alpha}, \| f \|_{\llip^r_\alpha}\big\}.
\end{eqnarray*}
We will also denote such seminorms and their associated spaces by $\LLip_\alpha^{r \wedge c}$ and $\llip_\alpha^{r \wedge c}$ whenever we want to make the explicit joint row and column norm salient. The proof of the fact that they are seminorms follows the same steps of \cite[Section 2]{JunMei2012BMO}. We include it here for the sake completeness.

\begin{proposition}
  \label{prp:CampanantoSeminorm}
  The quantities introduced in Definition \ref{def:CampanatoNorms} satisfy the triangular inequality. Furthermore their nulspaces coincide with $\ker(A^\frac12) \subset \M$.
\end{proposition}

\begin{proof}
  Let us start with $\llip_\alpha^c$. Since $P_s$ is completely positive we can define a $\M$-valued inner product on $\M \algtensor \M$ given by
  \[
    \big\langle f_1 \otimes g_1, f_2 \otimes g_2 \big\rangle_{P_s} = f_1^\ast P_s \big( g_1^\ast g_2 \big) f_2 
  \]
  Now, using the theory described in \cite[Chapter 5]{Lance1995} we can define a $W^\ast$-Hilbert $\M$-module $\M \otimes_{P_s} \M$  by completing with respect to the weak topology given by composing $\xi \mapsto \langle \xi, \xi \rangle^\frac12$ with any element $\varphi \in \M_\ast$.
  Now, notice that the function $\delta_s: \M \to \M \otimes_{P_s} \M$ given by $\delta_s(f) = P_s(f) \otimes \1  - \1 \otimes f$ satisfies that
  \[
    \big\langle \delta_s(f), \delta_s(f) \big\rangle_{P_s} = P_s|f|^2 - |P_s f|^2,
  \]
  therefore
  \[
    \| f \|_{\llip_\alpha^c} = \sup_{s > 0} \Big\{ s^{-\alpha} \, \big\| \delta_s(f) \big\|_{\M \otimes_{P_s} \M} \Big\}
  \]
  is a seminorm. The nulspace is defined by those operators such that $P_s|f|^2 = |P_s f|^2$ which, by Choi's theorem, see \cite[1.5.7]{BroO2008}, is given by the multiplicative domains $P_s(f \, g) = P_s(f) \, P_s(g)$ and, using Stinespring's Theorem, is given by the set of $f$ such that $P_s(f) = f$, which as discussed before definition \ref{def:Homogeneous} is given by $\ker(A^\frac12)$.
  
  The result is easier for $\LLip_\alpha^c$ . First note that $\| P_s |f|^2 \|_\infty^\frac12$ satisfies the triangular inequality for every $s$. Indeed
  \[
    \big\| P_t |f|^2 \big\|_\infty^\frac12 = \sup_{\varphi \in \Ball(\M_\ast)} \varphi \circ P_s \big( f^\ast \, f \big)^\frac12,
  \]
  but we can regard the second expression for each $\varphi$ as Hilbert space norm. It is clear that composing that norm with $f - P_sf$ preserves the triangular inequality. That the nulspace is given by the kernel is immediate.
\end{proof}

Like we did before, the spaces $\llip_\alpha^c$ and $\LLip_\alpha^c$ will be defined as weak closures of $\M^\circ = \M/\ker(A^\frac12)$, see \ref{def:Homogeneous}. First note that for each $s > 0$ and $0 \leq \alpha < 1$ we have a $W^\ast$-Hilbert $\M$-module $\M \otimes_{s^{-2 \alpha} P_s} \M$ whose inner product is given by scaling that of $\langle \cdot , \cdot \rangle_{P_s}$ with a factor $s^{-2 \alpha}$, i.e.
\[
  \langle \xi, \xi \rangle_{s^{-2 \alpha} P_s} = s^{-2 \alpha} \, \langle \xi, \xi \rangle_{P_s}.
\]
Equivalently you can see $\M \otimes_{P_s}^\alpha \M$ as the Hilbert $\M$-module associated with the completely positive map $s^{-\alpha} P_s$. Now, consider the $L^\infty$-direct integral of all such Hilbert modules with respect to $s > 0$, we will denote such space as
\[
  \M \otimes_{\Ps, \alpha} \M = L^\infty \big(\RR_+; \M \otimes_{s^{-2 \alpha} P_s} \M \big).
\]
Clearly $\M \otimes_{\Ps, \alpha} \M$ is a $W^\ast$-Hilbert $L^\infty(\RR_+;\M)$-module. Let $\delta$ be the  unbounded map given by the direct sum of all $\delta_s: \M \to \M \otimes_{s^{-2\alpha} P_s} \M$. We obtain a map $\delta_\alpha: \dom(\delta_\alpha) \subset \M \to \M \otimes_{\Ps, \alpha} \M$ which is weak-$\ast$ closable. We are going to use the map above to define a weak-$\ast$ topology with respect to which it is possible to take closures. Indeed, by the theory developed in \cite{JunSher2005} $\M \otimes_{\Ps, \alpha} \M$ is a dual space whose predual is given by a Hilbert $L^1(\RR_+; \M_\ast)$-module. The pullback of this weak-$\ast$ topology can be used to define $\llip_\alpha^c$ as follows
\begin{definition}
  \label{def:CampanatoWeak}
  We define $\llip_\alpha^c$ as the closure of $q[\dom(\delta_\alpha)] \cap \M^\circ$ with respect to the topology generated by the maps
  \[
    \big[ f + \ker(A^\frac12) \big] \longmapsto \big\langle \delta_\alpha(f), \varphi \big\rangle,
  \]
  for every $\varphi$ in the predual of $\M \to \M \otimes_{\Ps, \alpha} \M$ and where
  $\delta_\alpha: \dom(\delta_\alpha) \subset \M \to \M \otimes_{\Ps, \alpha} \M$ is defined as above.
  Similarly $\LLip_\alpha^c$ is defined as the closure of $\M^\circ$ with respect to the weak-$\ast$ topology given by pairing
  \[
    s \mapsto s^{-\alpha} \, \big( P_s |f - P_s f|^2 \big)^\frac12
  \]
  with elements in $L^1(\RR_+;\M_\ast)$.
  
  The construction of $\llip_\alpha = \llip_\alpha^c \cap \llip_\alpha^r$ and
  $\LLip_\alpha = \LLip_\alpha^c \cap \LLip_\alpha^r$ are performed similarly.
\end{definition}

It is worth noticing that the definition above makes perfect sense in the case of $\alpha = 0$ and that we recover the bounded mean oscillation spaces $\bmo^c(\Ps)$ and $\BMO^c(\Ps)$ defined in \cite{JunMei2012BMO}.
The definitions above give spaces $\LLip_\alpha^\dagger$, where $\dagger \in \{c, r, r \wedge c\}$ as Banach spaces but it is possible to give $\LLip_\alpha^c$ an operator space structure just by defining the following sequence of matrix seminorms
\begin{equation}
  \label{eq:ossLipAlpha}
  \big\| [f_{i j}] \big\|_{M_m[\LLip_\alpha^c]}
  =
  \sup_{s > 0} \left\{ s^{-\alpha} \, \Big\| (\Id \otimes P_s) \big| [f_{i j}] - [P_s f_{i j}] \big|^2 \Big\|_{M_m[\M]}^\frac12 \right\}.
\end{equation}
I.e. we identify $M_m[\LLip_\alpha^c]$ with the $\LLip_\alpha^c$ associated with the von Neumann algebra $M_m[\M] = M_m(\CC) \otimes \M$ and the Poisson semigroup given by $(\Id \otimes P_s)_{s \geq 0}$ which is subordinated to $(\Id \otimes T_t)$. The same can be done in the row case and in the symetric case. It is trivial to check that the matrix norms above satisfy Ruan's axioms.

\begin{remark}
  \label{rmk:DefCampanato}
  A few remarks are in order
  \begin{enumerate}[leftmargin=1cm, label = {\bf (\roman*)}]
    \item In the case in which $\M = L^\infty(\Omega,\mu)$ is an Abelian von Neumann algebra the column and row norms for  the $\LLip_\alpha$-spaces coincide. Nevertheless the matrix norms in \eqref{eq:ossLipAlpha} give a priori nonisomorphic operator space structures in the row and column cases and symmetric cases. It is relative simple to see that the the symmetric operator space structure is nonisomorphic to the column and the row ones. Indeed, given an operator space $E$ we can define its opposite $E^\op$ as the the operator space that have the same underlying Banach space but with the matrix norms given by
    \[
      \big\| [f_{i j}] \big\|_{M_m[E^\op]} = \big\| [f_{j i}] \big\|_{M_m[E]}.
    \]
    An operator space is called symmetric is $E^\op \cong E$. The space $\LLip_\alpha^{r \wedge c}$ is symmetric while $\LLip_\alpha^{c}$ and $\LLip_\alpha^{r}$ are not necessarily so.
    \item In contrast with the ease with which we could define an operator space structure for $\Lambda^\circ_\alpha(\Ts)$ and $\LLip_\alpha$ just by changing $\M$ by $M_m(\CC) \otimes \M$ and the semigroup $P_s$ by $\Id \otimes P_s$, we get that this approach is not suitable for $\llip_\alpha^c$ since the corresponding norms do not satisfy the Ruan's axioms in an obvious manner. It is unknown by the author whether there is direct definition of an operator space structure for $\llip_\alpha^c$ (even in the case of $\bmo^c(\Ps)$) that does not rely in comparisons of the norm of $\llip_\alpha$ with that of $\LLip_\alpha$ or in square function estimates.
    \item The spaces defined above are column and row analogues of the homogeneous H\"older classes $\Lambda_\alpha^\circ$. Although we have choose not to do so, we could have defined the analogues of the H\"older classes $\Lambda_\alpha$ as subsets of $\M$. The same arguments of Proposition \ref{prp:Duality} give that the corresponding classes are dual spaces. 
  \end{enumerate}
\end{remark}

In order to compare the norms $\LLip_\alpha$ and $\llip_\alpha$ with those that of $\Lambda_\alpha^\circ$ we will need the following lemma, which goes along the lines of Lemma \ref{lem:GradPos} and is taken from \cite{JunMei2012BMO} and \cite{JunMe2010}. Before stating the lemma recall that we will denote by $\Gamma_{A^\frac12}$ the gradient form associated to the generator $A^\frac12$ of $\Ps = (P_s)_{s \geq 0}$. Like in Section \ref{sct:Riesz} we will denote by $\widehat{\Gamma}$ the space-time gradient. We have that

\begin{lemma}[{\bf \cite[Lemma 1.1]{JunMei2012BMO}}]
  \label{lem:DerJungeMei12}
  Let $\Ts = (T_t)_{t \geq 0}$ be a Markovian semigroup, $\Ps = (P_s)_{s \geq 0}$ its subordinated Poisson semigroup and $\Gamma$ and $\Gamma_{A^\frac12}$. 
  \begin{enumerate}[leftmargin=1cm, label={\rm (\roman*)}]
    \item \label{itm:DerJungeMei12.1} For every $f \in \M$
    \[
      T_t |f|^2 - |T_t f|^2 = \int_0^t T_{t - s} \Gamma \big[ T_s f \big] \, d s.
    \]
    \item \label{itm:DerJungeMei12.2} For every $f \in \S$
    \[
      \Gamma_{A^\frac12}[f] = \int_0^\infty P_v \widehat{\Gamma} \big[ P_v f \big] \, d v,
    \]
    \item \label{itm:DerJungeMei12.3} As a consequence it holds that
    \[
      P_t \Gamma_{A^\frac12}[P_s f] = \int_0^\infty P_{t + v} \widehat{\Gamma} \big[ P_{v + s} f \big] \, d s,
    \]
  \end{enumerate}
\end{lemma}

\begin{proof}
  We have that \ref{itm:DerJungeMei12.3} follows trivially from \ref{itm:DerJungeMei12.2}. Thus, We shall only prove \ref{itm:DerJungeMei12.1} and \ref{itm:Cor.Subordination.2}. For \ref{itm:DerJungeMei12.1}, recall from the proof of proposition \ref{prp:KSBilinear} that $t \mapsto F_t$ given by
  \[
    F_t = T_{t} \big( |T_{s - t} f |^2 \big)
  \]
  is a positive and operator-increasing function satisfying that $F_0 =  |T_t f|^2$ and that $F_s = T_s|f|^2$. Therefore, by \eqref{eq:DerivBilinear} applied to the quadratic form $B[f] = f^\ast f$, we get
  \[
    T_t |f|^2 - |T_t f|^2 = \int_0^t \frac{d F_t}{d \, t} \, d s = \int_0^t T_{t - s} \Gamma \big[ T_s f \big] \, d s.
  \] 
  For \ref{itm:DerJungeMei12.2}, take again the operator-valued function $G_s$ from Lemma \ref{lem:GradPos} given by
  \[
    G_s = P_s \left( \frac{d P_s f}{d \, s}^\ast P_s f \right) + P_s \left( P_s f^\ast \frac{d P_s f}{d \, s} \right) - \frac{d P_s}{d \, s} \big( P_s f^\ast P_s f \big)
  \]
  and notice that 
  \begin{equation*}
    G_0 = \Gamma_{A^\frac12}[f]
  \end{equation*}
  and that, by Choi's theorem, $G_s \to 0$ weakly as $s \to \infty$, using the fundamental theorem of calculus we obtain the result.
\end{proof}

We are also going to use the following proposition. We will omit its proof since it is entirely contained in \cite{JunMei2012BMO}. All of the identities above are obtained from the iteration of \ref{itm:DerJungeMei12.2} and \ref{itm:DerJungeMei12.1} in the Lemma \ref{lem:DerJungeMei12} above to obtain that
\begin{equation}
  \label{eq:IteratedIy2}
   P_s |f|^2 - |P_s f|^2
   = \int_0^s P_{s - t} \Gamma_{A^\frac12} \big[ P_t f \big] \, d t
   = \int_0^s \int_0^\infty P_{s - t + v} \widehat{\Gamma} \big[ P_{t + v} f \big] \, d v d t
\end{equation}

\begin{proposition}[{\bf \cite[Proposition 2.5]{JunMei2012BMO}}]
  \label{prp:PointSquareJungeMei12}
  Let $\Ps = (P_s)_{s \geq 0}$ be as before, we have that
  \begin{enumerate}[leftmargin=1cm, label={\rm (\roman*)}, ref={\rm (\roman*)}]
    \item \label{itm:PointSquareJungeMei12.1}
    \[
       P_s |f|^2 - |P_s f|^2 = 2 \int_0^\infty \int_{\max\{ 0, t - s \}}^t P_{s - t + 2v} \widehat{\Gamma}[P_{t} f] \, d v \, d t.
    \]
    \item \label{itm:PointSquareJungeMei12.2}
    \[
      \int_0^\infty P_{s + t} \Big| \frac{d P_t}{d \, t}(f) \Big|^2 \min\{s, t\} \, d t
      \lesssim
      P_s |f|^2 - |P_s f|^2.
    \]
    \item \label{itm:PointSquareJungeMei12.3} If $\Gamma^2 \geq 0$ then
    \[
      \int_0^\infty P_{s + t} \widehat{\Gamma} \big[ f \big] \min\{s, t\} \, d t
      \lesssim
      P_s|f|^2 - |P_s f|^2
      \lesssim
      \int_0^\infty P_{\frac{s}{3} + t} \widehat{\Gamma} \big[ f \big] \min \Big\{ \frac{s}{3}, t \Big\} \, d t.
    \]
  \end{enumerate}
\end{proposition}

We can now proceed to see that the norms of $\LLip_\alpha$ and
$\llip_\alpha$ are bounded by those of $\Lambda_\alpha^\circ$.

\begin{proposition}
  \label{prp:ComparisonHolderCampanato}
  \
  \begin{enumerate}[leftmargin=1cm, label={\rm (\roman*)}, ref={\rm (\roman*)}]
    \item For every $0 < \alpha < 1$ we have that
    $\displaystyle{
      \| f \|_{\LLip_\alpha^\dagger} \lesssim \| f \|_{\Lambda^\circ_\alpha},
    }$
    where $\dagger \in \{r, c, r \wedge c \}$.
    \item If $\Ts$ has the $\Gamma^2 \geq 0$ property and $0 < \alpha < 1/2$, then 
    $\displaystyle{
      \| f \|_{\llip_\alpha^\dagger} \lesssim \| f \|_{\Lambda^\circ_\alpha},
    }$
    where $\dagger \in \{r, c, r \wedge c \}$.
  \end{enumerate}
\end{proposition}

\begin{proof}
  The first point is trivial. Indeed, 
  \begin{eqnarray*}
    \sup_{s > 0} \left\{ s^{-\alpha} \, \big\| P_s | f - P_s f|^2 \big\|_\infty^\frac12 \right\}
      & = & \sup_{s > 0} \left\{ s^{-\alpha} \, \left\| P_s \Big| \int_0^s \frac{d P_t}{d \, t}(f) \, d t  \Big|^2 \right\|_\infty^\frac12 \right\}\\
      & = & \sup_{s > 0} \left\{ s^{-\alpha} \, \int_0^s \Big\| \frac{d P_t}{d \, t}(f) \Big\|_\infty d t \right\}\\
      & \leq & \| f \|_{\Lambda_\alpha^\circ} \, \sup_{s > 0} \left\{ s^{-\alpha} \, \int_0^s t^{1 - \alpha} \, d t  \right\}
        \lesssim_{(\alpha)} \| f \|_{\Lambda_\alpha^\circ}.
  \end{eqnarray*}
  For the second one we need to use Proposition \ref{prp:PointSquareJungeMei12}. Notice that
  \begin{eqnarray*}
    \big\| P_s |f|^2 - |P_s f|^2 \big\|_\infty
      & =    & 2 \, \bigg\| \int_0^\infty \int_{\max\{ 0, t - s \}}^t P_{s - t + 2v} \widehat{\Gamma}[P_{t} f] \, d v \, d t \bigg\|_\infty \\
      & \leq & 2 \int_0^\infty \int_{\max\{ 0, t - s \}}^t \big\| P_{s - t + 2v} \widehat{\Gamma}[P_{t} f] \big\|_\infty d v \, d t \\
      & \leq & 2 \int_0^\infty \int_{\max\{ 0, t - s \}}^t \big\|\widehat{\Gamma}[P_{t} f] \big\|_\infty d v \, d t.
  \end{eqnarray*}
  By $\Gamma^2 \geq$ we can apply Proposition \ref{thm:RieszBound} to obtain that
  \begin{eqnarray*}
    \big\| P_s |f|^2 - |P_s f|^2 \big\|_\infty
      & \lesssim & \int_0^\infty \int_{\max\{ 0, t - s \}}^t \frac{\| f \|_{\Lambda_\alpha^\circ}^2}{t^{2 - 2 \alpha}} d v \, d t\\
      & \leq     & \| f \|_{\Lambda_\alpha^\circ}^2 \, \int_0^\infty  \frac{\min \{s, t \}}{t^{2 - 2 \alpha}} \, d t.\\
  \end{eqnarray*}
  But it is immediate that the integral above is bounded by
  \[
    \int_0^\infty  \frac{\min \{s, t \}}{t^{2 - 2 \alpha}} \, d t
    =
    \int_0^s \frac{1}{t^{1 - 2 \alpha}} \, d t + s \, \int_s^\infty  \frac{1}{t^{2 - 2 \alpha}} \, d t
    \lesssim \max \Big\{ 2 \alpha, \frac{1}{1 - 2 \alpha} \Big\} \, s^{2 \alpha}
  \]
  whenever $\alpha < 1/2$.
\end{proof}

It is a natural question whether the inverse inequality to those shown in the theorem above hold. In particular it seems like a very natural problem to ask whether $\LLip_\alpha^{r \wedge c}$ is equivalent to $\Lambda_\alpha^\circ$ up to constants. Notice also that, by Remark \ref{rmk:DefCampanato}, if we want the inequalities above to hold after matrix amplifications, that is to induce completely bounded maps, then the problem of comparing $\LLip_\alpha^{r \wedge c}$ and $\Lambda_\alpha^\circ$ becomes nontrivial even in the Abelian case.

Now, we will concern ourselves with the comparison between the norms of $\LLip_\alpha$ and $\llip_\alpha$.
The following proposition is a straightforward adaptation of \cite[Proposition 2.4(ii)/(iii)]{JunMei2012BMO}.

\begin{proposition}
  \label{prp:Eqnorm}
  \
  \begin{enumerate}[leftmargin=1cm, label={\rm (\roman*)}, ref={\rm (\roman*)}]
    \item For every $0 \leq \alpha < 1$ we have that
    \[
      \| f \|_{\LLip_\alpha^c} \leq (1 + 2^\alpha) \, \| f \|_{\llip_\alpha^c} + \sup_{s > 0} \Big\{ s^{-\alpha} \, \big\| P_s f - P_{2 s} f \big\|_\infty^\frac12 \Big\}
    \]
    \item If $\Ts$ has the $\Gamma^2 \geq 0$ property, for $0 \leq \alpha < 1/2$
    \[
      \| f \|_{\llip_\alpha^c} \leq \Big( 1 - 2^{\alpha - \frac12} \Big)^{-1} \| f \|_{\LLip_\alpha^c}
    \]
  \end{enumerate}
\end{proposition}

\begin{proof}
  For the first point add and subtract the term $|P_s(f - P_s)|^2$ inside the $\| \cdot \|_\infty$-norm
  \begin{eqnarray*}
      &      & s^{-\alpha} \, \big\| P_s |  - P_s f|^2 \big\|_\infty^\frac12\\
      & \leq & s^{-\alpha} \, \big\| P_s |f - P_s f|^2 - |P_s(f - P_s)|^2 \big\|_\infty^\frac12
               + s^{-\alpha} \, \big\| P_s f - P_{2 s} f \big\|_\infty\\
      & \leq & s^{-\alpha} \, \big\| P_s |f|^2 - |P_s f|^2 \big\|_\infty^\frac12
               + s^{-\alpha} \, \big\| P_s |P_s f|^2 - |P_{2 s} f|^2 \big\|_\infty^\frac12
               + s^{-\alpha} \, \big\| P_s f - P_{2 s} f \big\|_\infty.
  \end{eqnarray*}
  We have used the fact that $f \mapsto \| P_s |f|^2 - |P_s f|^2 \|_\infty^\frac12$ satisfies the triangular inequality. Now, taking suprema over $s > 0$ in both sides and doing the change of variable $s \mapsto s/2$ in the middle erm gives the result.
  For the second point we can proceed like in \cite[Proposition 2.4(iii)]{JunMei2012BMO}. Indeed
  \begin{eqnarray*}
    \big( P_s |f|^2 - |P_s f|^2 \big)^\frac12
      & =    & \bigg( \int_0^s P_{s - t} \Gamma_{A^\frac12} \big[ P_t f \big] \, d t \bigg)^\frac12 \\
      & \leq & \bigg( \int_0^s P_{s - t} \Gamma_{A^\frac12} \big[ P_t(f - P_s f) \big] \, d t \bigg)^\frac12
               + \bigg( \int_0^s P_{s - t} \Gamma_{A^\frac12} \big[ P_{s + t}(f) \big] \, d t \bigg)^\frac12 \\
      & =    & \big( P_s |f - P_s f|^2 - |P_s(f - P_s f)|^2 \big)^\frac12 
               + \bigg( \int_0^s P_{s - t} \Gamma_{A^\frac12} \big[ P_{s + t}(f) \big] \, d t \bigg)^\frac12\\
      & \leq & \big( P_s |f - P_s f|^2 \big)^\frac12 
               + \bigg( \int_0^s P_{s - t} \Gamma_{A^\frac12} \big[ P_{s + t}(f) \big] \, d t \bigg)^\frac12\\
      & \leq & \big( P_s |f - P_s f|^2 \big)^\frac12 
               + \bigg( \int_0^s P_{2 s - 2 t} \Gamma_{A^\frac12} \big[ P_{2 s}(f) \big] \, d t \bigg)^\frac12\\
      & =    & \big( P_s |f - P_s f|^2 \big)^\frac12 
               + \frac{1}{\sqrt{2}} \, \big( P_{2 s} |f|^2 - |P_{2 s} f|^2 \big)^\frac12\\
  \end{eqnarray*}
  Multiplying by $s$ and doing a change of variable, we obtain that
  \[
    \big( P_s |f|^2 - |P_s f|^2 \big)^\frac12
    \leq
    \big( P_s |f - P_s f|^2 \big)^\frac12
    + \frac{2^{\alpha}}{\sqrt{2}} \, (2 s)^{-\alpha} \big( P_{2 s} |f|^2 - |P_{2 s} f|^2 \big)^\frac12
  \]
  and when $\alpha < 1/2$ we can take suprema in $s$ and obtain the bound for the norm of $\llip_\alpha^c$.
\end{proof}

In order to bound the quantity $\| P_s f - P_{2 s} f \|$ we need to introduce the following seminorms, inspired when $\alpha = 0$,
by semigroup analogues of the Carleson measure.

\begin{definition}
  \label{def:CarlesonDecay}
  Let $f \in \S$, the test algebra, we define
  \begin{eqnarray}
    \| f \|_{\llip_\alpha^c(\partial)}
      & = & \sup_{s > 0} \bigg\{ s^{-\alpha} \, \bigg\| \, P_s \int_0^s \Big| t \, \frac{d P_s}{d \, t}(f) \Big|^2 \, \frac{d \, t}{t} \, \bigg\|_\infty^\frac12 \bigg\}\\
    \| f \|_{\llip_\alpha^c(\Gamma)}
      & = & \sup_{s > 0} \bigg\{ s^{-\alpha} \, \bigg\| \, P_s \int_0^s \Gamma \big[ t \, P_t f \big] \, \frac{d \, t}{t} \, \bigg\|_\infty^\frac12  \bigg\}\\
    \| f \|_{\llip_\alpha^c(\widehat{\Gamma})}
      & = & \sup_{s > 0} \bigg\{ s^{-\alpha} \, \bigg\| \, P_s \int_0^s \widehat{\Gamma} \big[ t \, P_t f \big] \, \frac{d \, t}{t} \, \bigg\|_\infty^\frac12 \bigg\}
  \end{eqnarray}
\end{definition}

Seen that the quantities above induce seminorms is immediate in the case of $\llip_\alpha^c(\partial)$, for the other two we need to repeat an argument that similar to those of Proposition \ref{prp:CampanantoSeminorm} and to the one before Lemma \ref{lem:Quadratic}. We omit the details.

Observe that the quantities above, when $\alpha = 0$, recover the Carleson measure norms for semigroup-type $\BMO$-spaces introduced in \cite[eq. (2.4)-(2.6)]{JunMei2012BMO}. The following proposition is a routine adaptation of \cite[Lemma 2.7]{JunMei2012BMO}. It serve to compare the seminorms introduced in the definition above with those of Proposition \ref{prp:PointSquareJungeMei12}. Recall that $\S$ is the test algebra introduced in Section \eqref{sct:Riesz}.

The propositon bellow follows by calculations identical to those in \cite[Lemma 2.7]{JunMei2012BMO}.

\begin{proposition}
  \label{prp:SqrFunctions}
  Let $B: \S \times \S \to \S$ be a (potentially unbounded) densely defined bilinear form and let $B[f]$ be its associated quadratic form. If $B' \geq 0$ we have that
  \begin{equation*}
    \sup_{s > 0} \bigg\{ s^{-\alpha} \, \bigg\| \, P_s \int_0^s B \big[ t \, P_t f \big] \, \frac{d \, t}{t} \, \bigg\|_\infty^\frac12 \bigg\}
    \sim_{(\alpha)}
    \sup_{s > 0} \bigg\{ s^{-\alpha} \, \bigg\| \int_0^\infty B \big[ P_{t + s} f \big] \min \{ s, t \} \, d t \, \bigg\|_\infty^\frac12 \bigg\}
  \end{equation*}
\end{proposition}

Observe that as a corollary we obtain, taking $B$ to be $\widehat{\Gamma}$, $\Gamma$ or $B[f] = (A^\frac12f^\ast)(A^\frac12 f)$, an alternative characterizations of the seminorms in Definition \ref{def:CarlesonDecay}. In the last case $B[P_t f] = |\partial_t P_t f|^2$ and, since the Poisson semigroup generated by $\partial_{tt}^2$ in $\RR_+$, satisfies that $\Gamma^2_{\partial_{tt}^2} \geq 0$, we can apply the equivalence of norms bellow
\begin{equation*}
  \sup_{s > 0} \bigg\{ s^{-\alpha} \, \bigg\| \, P_s \int_0^s \Big| t \, \frac{d P_t}{d \, t}(f) \Big|^2 \, \frac{d \, t}{t} \, \bigg\|_\infty^\frac12 \bigg\}
  \sim_{(\alpha)}
  \sup_{s > 0} \bigg\{ s^{-\alpha} \, \bigg\| \int_0^\infty \Big| t \, \frac{d P_t}{d \, t}(f) \Big|^2 \min \{ s, t \} \, d t \, \bigg\|_\infty^\frac12 \bigg\}
\end{equation*}
without any further hypotheses on the semigroup. In the case of $B = widehat{\Gamma}$ or $B = \Gamma$, we need to impose the $\Gamma^2 \geq 0$ property on the semigroup. Observe that, as a corollary of \ref{prp:SqrFunctions}, \ref{itm:PointSquareJungeMei12.2} and \ref{itm:PointSquareJungeMei12.2} in Proposition \ref{prp:PointSquareJungeMei12}, we obtain that

\begin{corollary}
  Let $\Ts$ be a Markovian semigroup and $\Ps = (P_s)_{s \geq 0}$ be its subordinated Markovian semigroup.
  \begin{enumerate}[leftmargin=1cm, label={\rm (\roman*)}, ref={\rm (\roman*)}]
    \item If $\Ts$ has the $\Gamma^2 \geq 0$ property, then $\displaystyle{ \| f \|_{\llip_\alpha^c} \sim \| f \|_{\llip_\alpha^c(\widehat{\Gamma})}}$.
    \item $\displaystyle{ \| f \|_{\llip_\alpha^c(\partial)} \lesssim \| f \|_{\llip_\alpha^c}}$
  \end{enumerate}
\end{corollary}

The last ingredient that we are missing is a way of bounding $s^{-\alpha} \, \| P_s f - P_{2s} f \|_\infty$ by any of the norms above. That is offered by the following proposition, which follows by Kadison-Schwarz and Jensen's inequality 
\begin{proposition}[{\bf \cite[Lemma 2.8]{JunMei2012BMO}}]
  \label{prp:Delta}
  \[
    \big\| P_s f - P_{(1 + \delta) s} f \big\|_\infty
    \lesssim
    c(\delta) \, \bigg\| P_s \int_0^s \Big| \frac{d P_t}{d \, t}(f) \Big|^2 \, t \, d t \bigg\|_\infty^\frac12
  \]
  where $c(\delta)$ behaves like
  \[
    c(\delta)
    \lesssim
    \begin{cases}
      \delta^\frac12 & \mbox{ when } 0 \leq \delta \leq 1 \\
      \big( 1 + \log_{\frac{3}{2}} \delta \big) & \mbox{ otherwise }
    \end{cases}
  \]
\end{proposition}

The following theorem follows as a consequence of the previous results

\begin{theorem}
  For every $0 \leq \alpha < 1$, we have that 
  \[
    \| f \|_{\LLip_\alpha^c} \lesssim \| f \|_{\llip_\alpha^c}.
  \]
  When $\Gamma^2 \geq 0$, $\llip_\alpha^c = \llip_\alpha^c(\widehat{\Gamma})$ with equivalent norms.
  Furthermore, if, on top of the  $\Gamma^2 \geq 0$, $0 < \alpha < 1/2$,
  then $\LLip_\alpha^c = \llip_\alpha^c$ with equivalent norms.
\end{theorem}
  
\begin{proof}
  The only missing piece is that $\| f \|_{\LLip_\alpha^c} \lesssim \| f \|_{\llip_\alpha^c}$. By Proposition \ref{prp:Eqnorm}(i) we
  have that
  \[
    \| f \|_{\LLip_\alpha^c}
    \leq (1 + 2^\alpha) \, \| f \|_{\llip_\alpha^c} + \sup_{s > 0} \Big\{ s^{-\alpha} \, \big\| P_s f - P_{2 s} f \big\|_\infty^\frac12 \Big\}
  \]
  But, by the Proposition \ref{prp:Delta} the second term is smallest than $\| f \|_{\llip_\alpha^c(\partial)}$ which is in turn bounded by $\| f \|_{\llip_\alpha^c(\widehat{\Gamma})}$ by the definition of $\widehat{\Gamma}$. But, since we are assuming $\Gamma^2 \geq 0$ this last seminorm is comparable to that of $\llip_\alpha^c$ and with that we conclude.
\end{proof}

\subsection*{Open Problems}
\begin{enumerate}[align=left, leftmargin=1.5cm, label={\bf (P.\arabic*)}, ref={\bf (P.\arabic*)}]
  \item \label{itm:Problem1} 
  In \cite{Weaver1996Derivation, Weaver2000DerivationII} Lipschitz algebras over noncommutative von Neumann algebras are studied as the domains of weak-$\ast$ closable derivations between von Neumann algebras. The semigroup H\"older spaces $\Lambda_\alpha(\Ts)$ with transference semigroups in the cases of quantum torii and quantum Euclidean spaces are indeed isomorphic to the graph of an unbounded derivation into a von Neumann algebra, as was shown in \cite{Weaver1998alphaLip}. This is also the case for complete Riemannian manifolds with $\Ric \geq 0$. But it is unclear if more general semigroup H\"older classes can be expressed as domains of derivations or by the contrary, there are examples of semigroup H\"older classes not associated with closable derivations. 
  \item \label{itm:Problem2} 
  Let $(X,d)$ be a metric space. In the case of classical H\"older classes it is immediate that if $\varphi: \CC \to \CC$ is a bounded $\beta$-H\"older continuous function, then the composition operator $T_\varphi(f) = \varphi \circ f$ is continuous $T_\varphi: \Lambda_\alpha(X,d) \to \Lambda_{\beta \alpha}(X,d)$ and satisfies that 
  \[
    \| T_\varphi(f) \|_{\Lambda_{\alpha \beta}} \leq \| \varphi \|_{\Lambda_\beta} \, \| f \|_{\Lambda_\alpha}^\beta.
  \]
  It is unknown to the author whether the same holds for semigroup H\"older classes. In particular, if $\Lambda_\alpha(\Ts)_{\mathrm{nor}}$ is the subset of normal operators of $\Lambda_\alpha(\Ts)$, when does the functional calculus associated to $\varphi$ induce an operator $T_\varphi: \Lambda_\alpha(\Ts)_{\mathrm{nor}} \to \Lambda_{\alpha \beta}(\Ts)_{\mathrm{nor}}$.
  \item \label{itm:Problem3} 
  When $(X,d)$ is a complete metric space, the classical H\"older classes over $X$ are not just dual spaces but double duals. Its double predual is given by the \emph{little H\"older classes}, i.e. the classes of those functions such that $|f(x) - f(y)| \in o(d(x,y)^\alpha)$, see \cite{Bade1987Amenability} and \cite[Chapter 4]{Weaver2018Book}. In our semigroup context there is a natural candidate to the double predual given by
  \[
    \lambda_\alpha(\Ts)
    = \bigg\{ f \in \Lambda_\alpha(\Ts)
             : \lim_{s \to 0^+} \Big\| s^{1 -\alpha} \frac{d P_s}{d \, s}(f) \Big\|_\infty = 0  \bigg\}
    \subset \Lambda_\alpha(\Ts).
  \]
  It is a natural question whether $\lambda_\alpha(\Ts)^{\ast \ast} \cong \Lambda_\alpha(\Ts)$
  \item \label{itm:Problem4}
  Similarly, the predual of the H\"older classes over $\RR^n$ have a natural characterization as Hardy spaces $H_p$, with $p < 1$. This sort of spaces and their properties, including generalizations of classical atomic decompositions, have been studied in more general context like those of doubling metric measure spaces. There is a real posibility of establishing a parallel theory in the case of semigroups by studying the preduals of the classes $\Lambda_\alpha(\Ts)$ here defined. This would constitute a theory analogous to the semigroup $H_1$-$\BMO$ duality established by Mei in \cite{Mei2012H1BMO}
  \item \label{itm:Problem5}
  Showing when $\LLip_\alpha^{r \wedge c}$ and $\Lambda_\alpha^\circ$ are isomorphic is a natural open problem. There are many consequences that will follow from that. For instance, it would be immediate that a interpolation result with $\BMO$-spaces would hold true. Indeed, since the spaces $\LLip_\alpha^\dagger$ form a scale for the complex interpolation method and $\LLip_0^\dagger = \BMO^\dagger(\Ps)$, we would have
  \[
    \big[ \Lambda_\alpha^\circ(\Ts), \BMO^{r \wedge c}(\Ps)\big]^\theta \cong \Lambda_{\alpha\theta}^\circ(\Ts).
  \]
  The Campanato formulas could also be useful for making problems on the boundedness of non-spectral singular integral operators more tractable.  
\end{enumerate}

\

\textbf{Acknowledgements:} The author is thankful for some discussions with Javier Parcet that happened during the ICMAT School I of the ``Thematic Research Program: Operator Algebras, Groups and Applications to Quantum Information'' in March of 2019. The references \cite{Bade1987Amenability, Weaver2018Book} where communicated through mathoverflow question \texttt{323369} by Y. Choi and N. Weaver.

\bibliographystyle{alpha}
\bibliography{../bibliography/bibliography}

\
\
\
\

\hfill \noindent \textbf{Adri\'an Gonz\'alez-P\'erez} \\
\null \hfill K U Leuven - Departement wiskunde \\ 
\null \hfill 200B Celestijnenlaan, 3001 Leuven. Belgium 
\\ \null \hfill\texttt{adrian.gonzalezperez@kuleuven.be}

\end{document}